\documentclass[11pt,times]{article}
\pdfoutput=1
\usepackage[margin=1in]{geometry}
\usepackage{amsmath,amssymb,amsthm,amsfonts}
\usepackage{enumitem}
\usepackage{microtype}
\usepackage{url}
\usepackage[small,sf]{caption}
\usepackage[font=sf]{floatrow}
\usepackage{verbatim}
\usepackage{graphicx}
\usepackage{amsthm}
\usepackage{booktabs}
\usepackage[usenames,dvipsnames]{color}
\usepackage[colorlinks,linkcolor=blue]{hyperref}
\usepackage{chngcntr}
\counterwithin{table}{section}
\hypersetup{citecolor=blue}

\allowdisplaybreaks[1]
\numberwithin{equation}{section}

\makeatletter
\renewcommand{\l@section}{\@dottedtocline{1}{1.5em}{2.3em}}
\makeatother

\def\arXiv#1{\href{http://arXiv.org/abs/#1}{arXiv:#1}}

\newtheorem*{thmnn}{Theorem}
\newtheorem{thm}{Theorem}[section]
\newtheorem{prop}[thm]{Proposition}
\newtheorem{cor}[thm]{Corollary}

\theoremstyle{remark}
\newtheorem*{rmk}{Remark}

\newcommand{\sfcaption}[1]{\caption{{\sf #1}}}

\newcommand{\Tr}{\operatorname{Tr}}

\newcommand{\spmat}[1]{\left( \begin{smallmatrix} #1 \end{smallmatrix} \right)}
\newcommand{\jac}{\operatorname{Jac}}
\newcommand{\Disc}{\operatorname{Disc}}
\newcommand{\End}{\operatorname{End}}
\newcommand{\Mod}{\operatorname{Mod}}
\newcommand{\Sp}{\operatorname{Sp}}
\newcommand{\PSp}{\operatorname{PSp}}
\newcommand{\teich}{\operatorname{Teich}}

\newcommand{\pic}{\operatorname{Pic}}
\newcommand{\M}{\mathcal M}
\newcommand{\ord}{\mathcal O}

\newcommand{\pp}{\mathbb P}
\newcommand{\cc}{\mathbb C}

\newcommand{\hh}{\mathbb H}
\newcommand{\rr}{\mathbb R}
\newcommand{\qq}{\mathbb Q}

\newcommand{\zz}{\mathbb Z}
\newlength{\polywidth}
\def \polysize {\footnotesize}
\newcommand{\als}{\addlinespace[0.4em]}
\newcommand{\ppoly}[1]{\parbox[c]{6.25in}{\hangindent=2.0em \hangafter=1 $#1$} }

\newcommand{\wD}{w_D}
\newcommand{\bD}{b_D}

\DeclareMathOperator{\PSL}{PSL}

\begin{document}
\title{Algebraic models and arithmetic geometry\\ of Teichm\"uller curves in genus two}

\author{Abhinav Kumar and Ronen E. Mukamel}
\date{September 5, 2016}

\maketitle

\begin{abstract}
A Teichm\"uller curve is an algebraic and isometric immersion of an
algebraic curve into the moduli space of Riemann surfaces.  We give
the first explicit algebraic models of Teichm\"uller curves of
positive genus.  Our methods are based on the study of certain Hilbert
modular forms and the use of Ahlfors's variational formula to identify
eigenforms for real multiplication on genus two Jacobians.  We also
present evidence that Teichm\"uller curves admit a rich arithmetic
geometry by exhibiting examples with small primes of bad reduction and
notable divisors supported at their cusps.
\end{abstract}

\tableofcontents

\section{Introduction}
\label{sec:introduction}
Let $\M_g$ denote the moduli space of Riemann surfaces of genus $g$.
The space $\M_g$ can be viewed as an algebraic variety and carries a
natural Teichm\"uller metric.  An algebraic immersion of a curve into
moduli space
\[ f : C \to \M_g \]
is a {\em Teichm\"uller curve} if $C$ is biholomorphic to a finite
volume hyperbolic Riemann surface $\hh/\Gamma$ in such a way that $f$
induces a local isometry.  The first example of a Teichm\"uller curve
is the modular curve $\hh/\operatorname{PSL}_2(\zz) \to \M_1$.  Other
examples emerge from the study of square-tiled surfaces and billiards
in polygons \cite{veech:ngon}.  In genus two, each discriminant $D$ of
a real quadratic order determines a Weierstrass curve $W_D \to \M_2$
(defined below) which is a disjoint union of finitely many
Teichm\"uller curves \cite{mcmullen:billiards} (see also
\cite{calta:periodicity}).  The curve $W_D$ is related to billiards in
the $L$-shaped polygon described in Figure \ref{fig:LTable} and
Weierstrass curves are the main source of Teichm\"uller curves in
$\M_2$ \cite{mcmullen:torsion}.

Few explicit algebraic models of Teichm\"uller curves have appeared in
the literature.  The current list of examples
\cite{bouwmoeller:nonarithmetic,bouwmoeller:triangleveechgps,lochak:arithmetic}
consists of curves of genus zero and hyperbolic volume at most $3 \pi$
and is produced by a variety of ingenious methods which will likely be
difficult to extend.  In this paper, we describe a general method to determine algebraic models for the Weierstrass curves in genus two.  
We use our method to describe $W_D$ for each of the thirty fundamental discriminants $D < 100$.  These examples include a Teichm\"uller curve of genus eight and hyperbolic volume $60\pi$.  Our methods are based on the study of certain Hilbert modular forms and a technique for identifying eigenforms for real multiplication on genus two Jacobians based on Ahlfors's variational formula.
We expect that the resulting method to explicitly describe the action of real multiplication on the one-forms of a Riemann surface will be much more broadly applicable. In particular, it should be useful for studying real multiplication in higher genera, and the Prym Teichm\"uller curves in $\M_3$ and $\M_4$ \cite{mcmullen:prym}.

A growing body of literature demonstrates that Teichm\"uller curves
are exceptional from a variety of perspectives.  Teichm\"uller curves
have celebrated applications to billiards in polygons and dynamics on
translation surfaces \cite{veech:ngon,veech:billiards} and give
examples of interesting Fuchsian differential operators
\cite{bouwmoeller:nonarithmetic}.  The variations of Hodge structures
associated to Teichm\"uller curves have remarkable properties
\cite{moller:variation,moller:torsion} which suggest that
Teichm\"uller curves are natural relatives of Shimura curves. Each
Teichm\"uller curve, like a Shimura curve, is isometrically immersed
in a Hilbert modular variety equipped with its Kobayashi metric and is
simultaneously defined as an algebraic curve over a number field and
uniformized by a Fuchsian group defined over number field.  These
facts about Teichm\"uller curves have been used by several authors to
explain the lack of examples in higher genus
\cite{bainbridgemoller:deligne,matheuswright:htplanes}.  A secondary
goal of this paper is to present evidence drawn from our examples that
Teichm\"uller curves also admit a rich arithmetic geometry.  We show
in particular that many of our examples have small and orderly primes
of bad reduction and notable divisors supported at their cusps.

\paragraph{Weierstrass curves in Hilbert modular surfaces.}  
For each integer $D > 0$ with $D \equiv 0,1\bmod 4$, let $\ord_D$ be
the quadratic ring of discriminant $D$.  The {\em Hilbert modular
  surface of discriminant $D$} is the complex orbifold $X_D = \hh
\times \hh / \PSL(\ord_D \oplus \ord_D^\vee)$.\footnote{The surface
  $X_D$ is isomorphic to $\hh \times \overline{\hh} / \PSL_2(\ord_D)$
  and is typically denoted $Y_{-}(D)$ in the algebraic geometry
  literature \cite{vdgeer:hms,hirzebruchzagier:intersectionnos}.}
When viewed as an algebraic surface, $X_D$ is a moduli space of
principally polarized abelian surfaces with real multiplication by
$\ord_D$.  The {\em Weierstrass curve of discriminant} $D$ is the
moduli space $W_D$ consisting of pairs $(X,[\omega])$ where: (1) $X$
is a Riemann surface of genus two, (2) $\omega$ is a holomorphic
one-form on $X$ with double zero and (3) the Jacobian $\jac(X)$ admits
real multiplication by $\ord_D$ stabilizing the one-form up to scale
$[\omega]$.  The period mapping sending a Riemann surface to its
Jacobian lifts to an embedding of $W_D$ in $X_D$.

Explicit algebraic models of Hilbert modular surfaces are obtained in
\cite{elkieskumar:hms} by studying elliptic fibrations of K3 surfaces.
\begin{thmnn}[Elkies-Kumar]
For fundamental discriminants $1 < D < 100$, the Hilbert modular
surface $X_D$ is birational to the degree two cover of the
$(r,s)$-plane branched along the curve $\bD(r,s)=0$ where $\bD$ is the
polynomial in Table \ref{tab:bDrs}.
\end{thmnn}

\begin{figure}
  \begin{center}
   \includegraphics[scale=0.6]{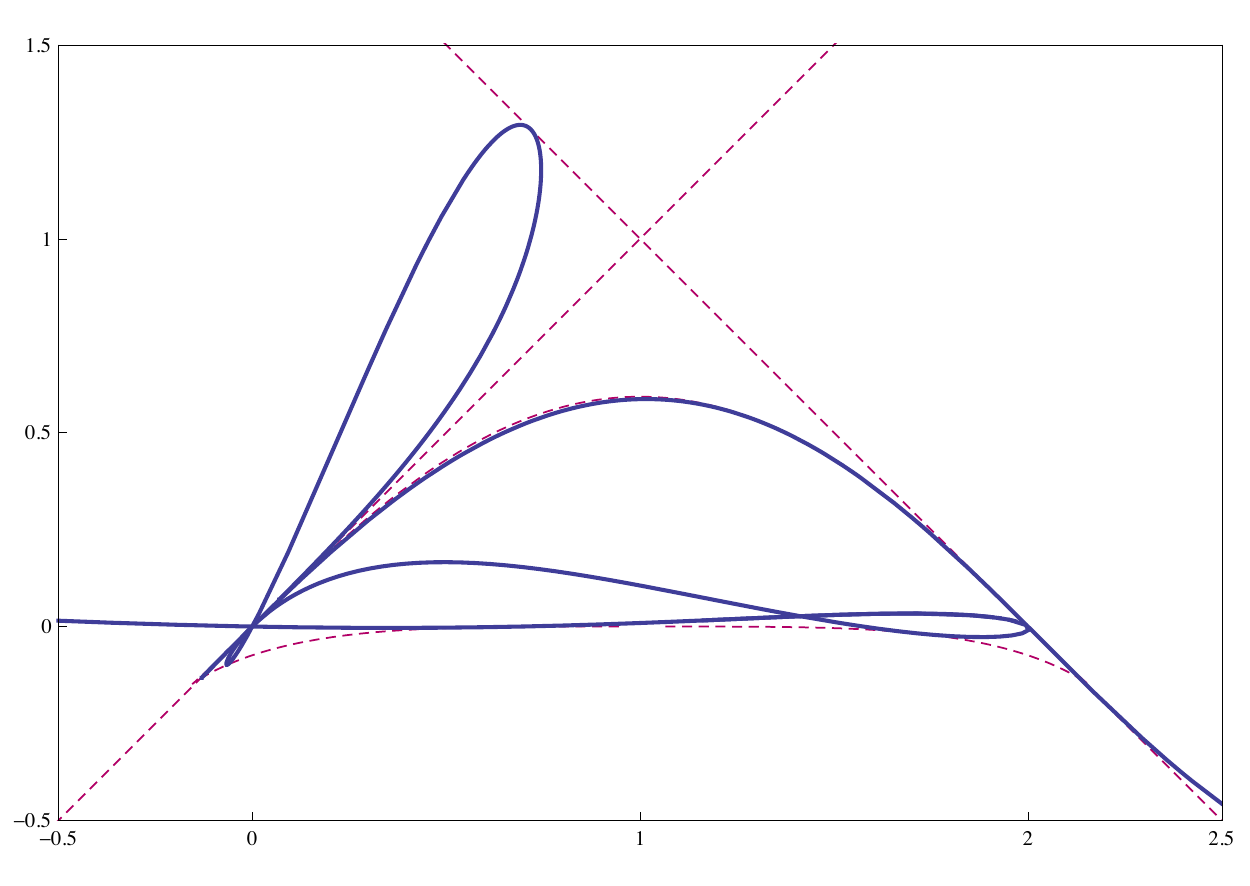}
\end{center}
\sfcaption{\label{fig:BW44} The Hilbert modular surface $X_{44}$ is
  birational to the degree two cover of the $(r,s)$-plane branched
  along the curve $b_{44}(r,s) =0$ (dashed).  The Weierstrass curve
  $W_{44}$ is birational to the curve $w_{44}(r,s)=0$ (solid).}
\end{figure}

The techniques used in \cite{elkieskumar:hms} also provide an algebraic description of the image of the rational map $X_D \to \M_2$.  However, these techniques do not readily adapt to describe the action of $\ord_D$ on the Jacobians in $X_D$.  
We address this challenge by developing a method for eigenform location, which
computes the action of $\ord_D$ on the space of holomorphic
one-forms on a Riemann surface whose Jacobian lies on $X_D$. We use
our method, which we describe briefly at the end of this section and more extensively in \S \ref{sec:eformverify}, to identify the locus corresponding to $W_D$ in the
model for $X_D$ above.
\begin{thm}
\label{thm:wDrs}
For fundamental discriminants $1 < D < 100$, the Weierstrass curve $W_D$ is birational to the curve
$\wD(r,s) = 0$ where $\wD$ is the polynomial in Table \ref{tab:wDrs}.
\end{thm}
The first Weierstrass curve of positive genus is the curve $W_{44}$ of
genus one.  The birational model $w_{44}(r,s)=0$ of $W_{44}$ is
depicted in Figure \ref{fig:BW44} along with the curve
$b_{44}(r,s)=0$.  Our proof of Theorem \ref{thm:wDrs} will yield an explicit birational model of the
universal curve over $W_D$ for fundamental discriminants $1 < D <
100$.

\paragraph{Spin components of Weierstrass curves.}  The curve $W_D$ is irreducible except when $D \equiv 1 \bmod 8$ in which case $W_D = W_D^0 \sqcup
W_D^1$ is a disjoint union of two irreducible components distinguished
by a spin invariant \cite{mcmullen:spin}. For such discriminants, the
components of $W_D$ have Galois conjugate algebraic models defined
over $\qq(\sqrt{D})$ \cite{bouwmoeller:nonarithmetic}.  Our next theorem distinguishes
the components of reducible $W_D$ in the models given in Theorem \ref{thm:wDrs}.
\begin{thm}
\label{thm:wDers}
For fundamental discriminants $1 < D < 100$ with $D \equiv 1 \bmod 8$,
the curve $W_D^\epsilon$ is birational to the curve $\wD^\epsilon(r,s)
= 0$ where $\wD^0$ is the polynomial in Table \ref{tab:wD0rs} and
$\wD^1$ is the Galois conjugate of $\wD^0$.
\end{thm}
\begin{figure}
  \begin{center}
    \includegraphics[scale=0.25]{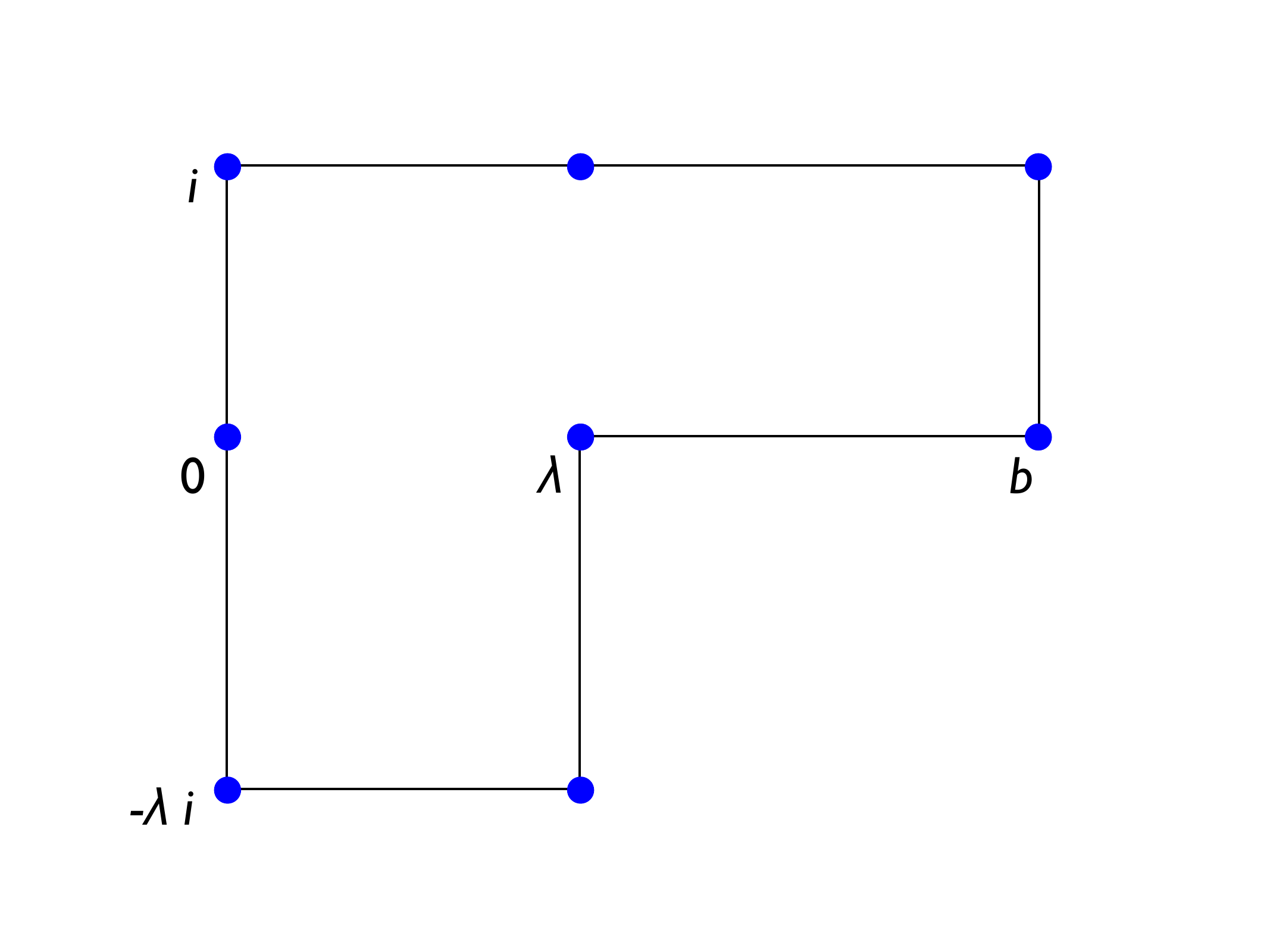}
\end{center}
\sfcaption{\label{fig:LTable} {\sf The Weierstrass curve $W_D$ emerges
    from the study of billiards in an $L$-shaped polygon obtained
    from a $\lambda$-by-$\lambda$ square and a $b$-by-$1$ rectangle
    where $\lambda = (e+\sqrt{D})/2$, $b =(D-e^2)/4$ and $e = 0$ or
    $-1$ with $e \equiv D \bmod 2$.}}
\end{figure}

\paragraph{Rational, hyperelliptic and plane quartic models.}  
The polynomials $\wD$ listed in Table \ref{tab:wDrs} are complicated
in part because they reflect how $W_D$ is embedded in $X_D$.  The
homeomorphism type of $W_D$ is determined in
\cite{bainbridge:eulerchar,mcmullen:spin,mukamel:orbpts} and in Table
\ref{tab:homeotype} we list the homeomorphism type of $W_D$ for the
discriminants considered in this paper.  For fundamental discriminants
$D \leq 73$ with $D \neq 69$, the irreducible components of $W_D$ have
genus at most three and algebraic models simpler than those given by
Theorems \ref{thm:wDrs} and \ref{thm:wDers}.

For discriminants $D \leq 41$, each irreducible component of $W_D$ has
genus zero.  Our proof of Theorems \ref{thm:wDrs} and \ref{thm:wDers}
will give rational parametrizations of the irreducible components of
$\wD(r,s)=0$ for such $D$ and yield our next result.
\begin{thm}
\label{thm:rationalwd}
For fundamental discriminants $D \leq 41$, each component of $W_D$ is
birational to $\mathbb P^1$ over $\qq(\sqrt{D})$.  For $D \leq 41$
with $D \not \equiv 1 \bmod 8$ and $D \neq 21$, the curve $W_D$ is
also birational to $\mathbb P^1$ over $\qq$.  The curve $W_{21}$ has
no rational points and is birational over $\qq$ to the conic
$g_{21}(x,y) = 0$ where:
\[ g_{21}(x,y) = 21 \left(11 x^2-182 x-229\right)+y^2.\]
\end{thm}
The curve $W_{44}$ of genus one and the curves $W_{53}$ and $W_{61}$
of genus two are hyperelliptic and the curves $W_{56}$ and $W_{60}$ of
genus three are canonically embedded as smooth quartics in $\mathbb
P^2$.  Our next theorem identifies hyperelliptic and plane quartic
models of these curves.
\begin{thm}
\label{thm:gDxy}
For $D \in \left\{ 44, 53, 56, 60, 61 \right\}$, the curve $W_D$ is
birational to $g_D(x,y) = 0$ where $g_D$ is the polynomial listed in
Table \ref{tab:wpqmodels}.
\end{thm}
The irreducible components of $W_{57}$, $W_{65}$ and $W_{73}$ have
genus one.  We also identify hyperelliptic models of these curves.
\begin{thm}
\label{thm:gDexy}
For $D \in \left\{ 57, 65, 73 \right\}$, the curve $W_D^\epsilon$ is
birational to $g_D^\epsilon(x,y) = 0$ where $g_D^0$ is the polynomial
listed in Table \ref{tab:wpqmodels} and $g_D^1$ is the Galois
conjugate of $g_D^0$.
\end{thm}

\setlength{\polywidth}{14cm}
\def \polysize {\normalsize}
\begin{table}
\begin{tabular}{c}
\toprule
Hyperelliptic and plane quartic models of Weierstrass curves \\ \als
\midrule \als
\ppoly{g_{44}(x,y) = x^3+x^2+160 x+3188-y^2} \\ \als
\ppoly{g_{53}(x,y) = 7711875 + 3572389 x + 777989 x^2 + 100812 x^3 + 8252 x^4 + 401 x^5 + 9 x^6 - (1 + x^2) y - y^2} \\ \als
\ppoly{g_{56}(x,y) = 35 + 10 x - 20 x^2 - 2 x^3 + x^4 - 43 y + 15 x y + 5 x^2 y - x^3 y + 33 y^2 - x y^2 - 5 x^2 y^2 - 10 y^3 + 4 x y^3 + 4 y^4} \\ \als
\ppoly{g_{57}^0(x,y) =  x^3 + \frac{1}{2}(1+\sqrt{57}) x^2 + (12+\sqrt{57})x + \frac{1}{2} (31-\sqrt{57}) - \frac{1}{2} (1+\sqrt{57}) x y -y^2} \\ \als
\ppoly{g_{60}(x,y) = 4 x^4-8 x^3 y-4 x^3+50 x^2 y^2-2 x^2 y-44 x y^3-56 x y^2+10 x y+228 y^4-32 y^3-8 y^2+y} \\ \als
\ppoly{g_{61}(x,y) = 12717 - 527 x - 6117 x^2 + 1498 x^3 - 604 x^4 - 282 x^5 + 324 x^6 -\left(x^2+x+1\right) y-y^2} \\ \als
\ppoly{g_{65}^0(x,y) = x^3+\left(27 \sqrt{65}-229\right) x^2+\frac{1}{2} \left(11225 \sqrt{65}-90375\right) x-y^2} \\ \als
\ppoly{g_{73}^0(x,y) =  x^3 + \frac{1}{2} (1+\sqrt{73}) x^2 - \frac{1}{2} (701+83\sqrt{73}) x -52 + 36\sqrt{73}  - \frac{1}{2} (221+ 17\sqrt{73}) y -y^2 }\\ \als
\bottomrule
\end{tabular}
\sfcaption{\label{tab:wpqmodels} For discriminants $44 \leq D \leq 73$
  with $D \neq 69$, each irreducible component of $W_D$ has either a
  hyperelliptic or plane quartic model defined above (cf. Theorems
  \ref{thm:gDxy} and \ref{thm:gDexy}).}
\end{table} 

\paragraph{Arithmetic of Teichm\"uller curves.}  
We hope that the models of Weierstrass curves in Tables
\ref{tab:wpqmodels}, \ref{tab:wDrs} and \ref{tab:wD0rs} will encourage
the study of the arithmetic geometry of Teichm\"uller curves.  To that
end, we now list several striking facts about these examples that give
evidence toward the theme:
\begin{center}
{\em Teichm\"uller curves are arithmetically interesting.} 
\end{center}
We will denote by $\overline{W}_D$ the smooth, projective curve
birational to $W_D$. The curve $\overline{W}_D$ is obtained from $W_D$
by filling in finitely many cusps on $W_D$ (studied in
\cite{mcmullen:spin}) and smoothing finitely many orbifold points
(studied in \cite{mukamel:orbpts}).  Our rational, hyperelliptic and
plane quartic birational models of low genus components of $W_D$
extend to biregular models of components of $\overline{W}_D$.
Throughout what follows, we identify $\overline{W}_D$ with these
biregular models via the parametrizations given in auxiliary computer
files, as described in Section \ref{sec:arithmetic}.

\paragraph{Singular primes.}  
The first indication that the curves $\overline{W}_D$ have interesting
arithmetic is the fact our low, positive genus examples are singular
only at small primes.  Our next two theorems suggest the following.
\begin{center}{\em 
The primes of bad reduction for Teichm\"uller curves \\ have
arithmetic significance.}
\end{center} 
To formulate a precise statement, we define
\begin{equation}
\label{eqn:ND}
N(D) = 2 \cdot D \cdot\prod_{e} \frac{D-e^2}{4} \mbox{ where $e$
  ranges in } \left\{ e : e > 0, e \equiv D \bmod 2 \mbox{ and } e^2<D
\right\}.
\end{equation}
The quantity $N(D)$ bears a striking resemblance to formulas in the
arithmetic of ``singular moduli'' of elliptic curves
\cite{Gross-Zagier}.  The number $N(D)$ is also closely related to the
product locus $P_D \subset X_D$ parametrizing polarized products of
elliptic curves with real multiplication.  The curve $P_D$ is a
disjoint union of modular curves each of whose levels divide $N(D)$
(\cite{mcmullen:spin}, \S 2).  In particular, the primes of bad
reduction for $P_D$ all divide $N(D)$ \cite{Silverman-Elliptic}. For
many of our examples, we find that the same is true of the primes of
bad reduction for $\overline{W}_D$.
\begin{thm}
\label{thm:singularprimes}
For discriminants $D \in \left\{ 21, 44, 53, 56, 60, 61 \right\}$, the
curve $\overline{W}_D$ has bad reduction at the prime $p$ only if $p$
divides $N(D)$.
\end{thm}
For Weierstrass curves birational $\pp^1$ over $\qq$, we give explicit parametrizations of $\overline{W}_D$ by the projective $t$-line in the auxiliary files. We define the
{\em cuspidal polynomial} $c_D(t)$ to be the monic polynomial
vanishing simply at the cusps of $\overline{W}_D$ in the affine
$t$-line and nowhere else, and obtain the following genus zero analogue
of Theorem \ref{thm:singularprimes}.
\begin{thm}
\label{thm:cusppoly}
For $D \leq 41$ with $D \not \equiv 1 \bmod 8$ and $D \neq 21$, the
cuspidal polynomial $c_D(t)$ is in $\zz[t]$ and a prime $p$ divides
the discriminant of $c_D(t)$ only if $p$ divides $N(D)$.
\end{thm}
The primes of singular reduction for our models of low, positive genus
Weierstrass curves are listed in Table \ref{tab:singularprimes} and
the cuspidal polynomials for Weierstrass curves birational to $\pp^1$
over $\qq$ are listed in Table \ref{tab:cusppolys}. Note that it is conceivable that we could get a smaller set of bad primes by choosing a different parametrization.

\paragraph{Divisors supported at cusps.} 
The divisors supported at cusps of $\overline{W}_D$ provide further
evidence that Teichm\"uller curves are arithmetically interesting.
The Fuchsian groups presenting Teichm\"uller curves as hyperbolic
orbifolds are examples of Veech groups.  Our next three theorems
suggest that
\begin{center}
{\em Veech groups have a rich theory of modular forms.}
\end{center}
Veech groups uniformizing Teichm\"uller curves can be computed
by the algorithm described in \cite{mukamel:fundamentaldomains} and a
fundamental domain for the group uniformizing $W_{44}$ is depicted in
Figure \ref{fig:fundamentaldomain}.  For background on Veech groups
see e.g. \cite{masurtabachnikov:billiards,zorich:flatsurfaces}.
\begin{figure}
  \begin{center}
    \includegraphics[scale=0.5]{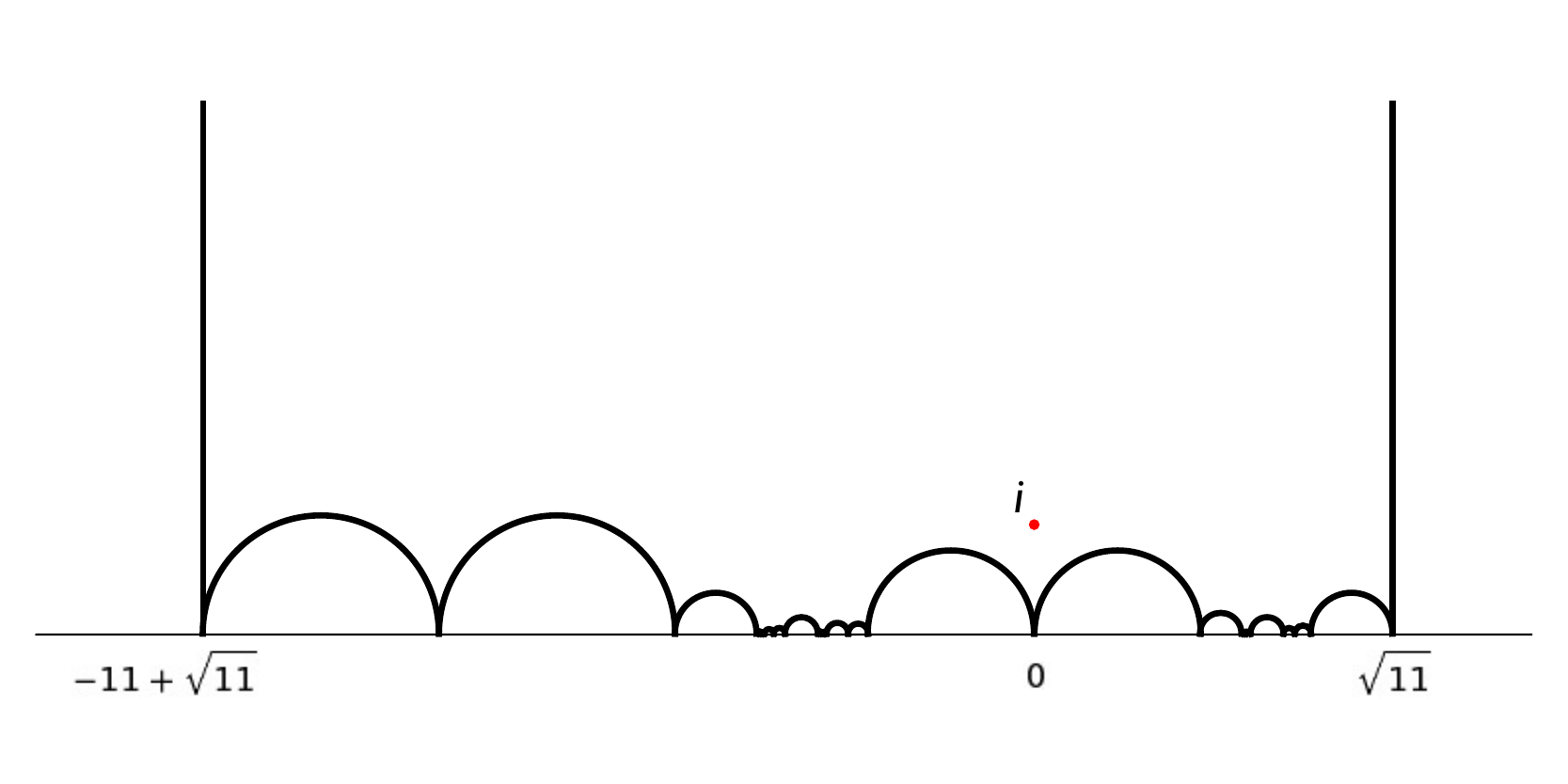}
\end{center}
\sfcaption{\label{fig:fundamentaldomain} {\sf The ideal polygon in
    $\hh$ depicted above is a fundamental domain for the Veech group
    uniformizing $W_{44}$.}}
\end{figure}

By the Manin-Drinfeld theorem
\cite{drinfeld:modularcurves,manin:parabolicpts}, the degree zero
divisors supported at the cusps of the modular curve $X_0(m) =
\overline{\hh / \Gamma_0(m)}$ generate a finite subgroup of the Picard
group $\pic^0(X_0(m))$.  The same is not quite true for divisors
supported at cusps of Weierstrass curves.
\begin{thm}
\label{thm:w44mw}
The subgroup of $\pic^0\left(\overline{W}_{44}\right)$ generated by
divisors supported at the nine cusps of $W_{44}$ is isomorphic to
$\zz^2$.
\end{thm}
While the cuspidal subgroup of $\overline{W}_{44}$ is not finite, it
is small in the sense that there are (many) principal divisors
supported at cusps.  In other words, there are non-constant regular
(algebraic) maps $W_{44} \to \cc^*$.  Several other Weierstrass curves
also enjoy this property.
\begin{thm}
\label{thm:principaldivisors}
Each of the curves $W_{44}$, $W_{53}$, $W_{57}$, $W_{60}$, $W_{65}$,
and $W_{73}$ admits a non-constant regular (algebraic) map to $\cc^*$.
\end{thm}
For several of the genus two and three Weierstrass curves, we also
find canonical divisors supported at cusps.
\begin{thm}
\label{thm:canonicaldivisors}
Each of the curves $\overline{W}_{53}$, $\overline{W}_{56}$ and
$\overline{W}_{60}$ has a holomorphic one-form vanishing only at
cusps.  The curve $\overline{W}_{61}$ has no holomorphic one-form
vanishing only at cusps.
\end{thm}
\noindent In Figure \ref{fig:G60}, the plane quartic model for
$\overline{W}_{60}$ is shown with the locations of the cusps marked.
The five dashed lines meet $\overline{W}_{60}$ only at cusps and each
corresponds to a holomorphic one-form up to scale on
$\overline{W}_{60}$ vanishing only at cusps.  The ratio of two such
forms corresponds to a holomorphic map $W_{60} \to \cc^*$.

\begin{figure}
  \begin{center}
   \includegraphics[scale=0.55]{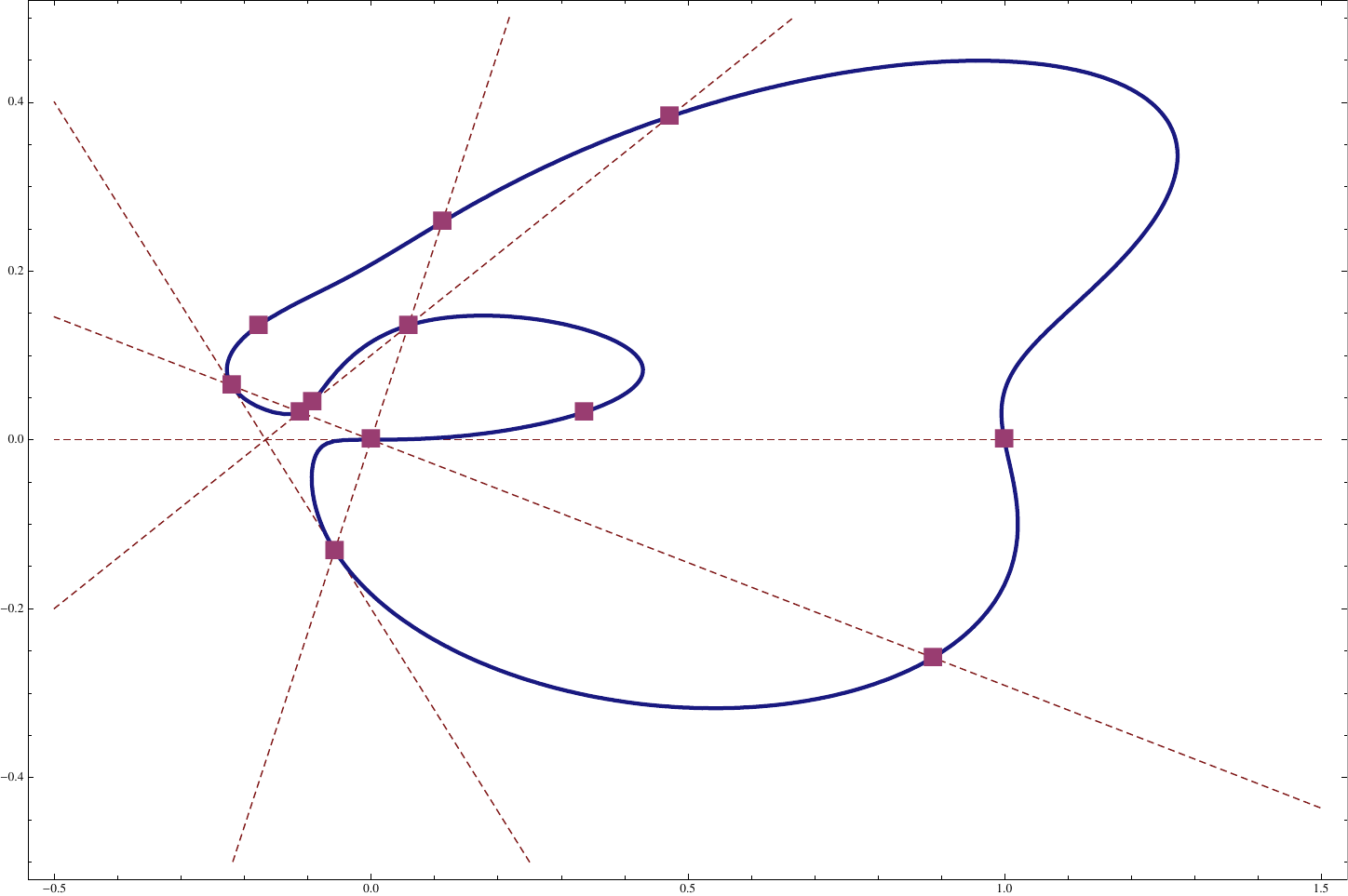}
\end{center}
\sfcaption{\label{fig:G60} {\sf The curve $\overline{W}_{60}$ is
    biregular to the plane quartic $g_{60}(x,y) = 0$ (solid) and the
    five lines shown (dashed) meet $\overline{W}_{60}$ only at cusps
    (squares).}}
\end{figure}

\paragraph{Numerical sampling and Hilbert modular forms.}  
As we now describe, the equations in Table \ref{tab:wDrs} were
obtained by numerically sampling the ratio of certain Hilbert modular
forms.  For $\tau =(\tau_1,\tau_2) \in \hh \times \hh$, define
matrices
\begin{equation}
\label{eqn:Btau}
 \Pi(\tau) = \begin{pmatrix} 
1 & \frac{D+\sqrt{D}}{2} & \tau_1 \frac{1+\sqrt{D}}{2} & \tau_1 \frac{\sqrt{D}}{D} \\ 
1 & \frac{D-\sqrt{D}}{2} & \tau_2 \frac{1-\sqrt{D}}{2} &\tau_2\frac{-\sqrt{D}}{D} 
\end{pmatrix} 
\mbox{ and } M = \frac{1}{2} \spmat{ D+\sqrt{D} & 0 \\ 0 & D-\sqrt{D} }.
\end{equation}
Since multiplication by $M$ preserves the lattice $\Pi(\tau) \cdot
\zz^4$, the abelian variety $B(\tau) = \cc^2/\left(\Pi(\tau)\cdot
\zz^4\right)$ admits real multiplication by $\ord_D$, and the forms
$dz_1$ and $dz_2$ on $\cc^2$ cover $\ord_D$-eigenforms $\eta_1(\tau)$
and $\eta_2(\tau)$ on $B(\tau)$.  There are meromorphic functions $a_k
: \hh \times \hh \to \cc$ for $0 \leq k \leq 5$ so that, for most
$\tau \in \hh \times \hh$, the Jacobian of the algebraic curve
\begin{equation}
\label{eqn:Ywm} 
Y(\tau) \in \M_2 \mbox{ birational to the plane curve } z^2 = w^6 +
a_5(\tau) w^5 + \dots + a_1(\tau) w + a_0(\tau)
\end{equation}
is isomorphic to $B(\tau)$ and the forms $\eta_1(\tau)$ and
$\eta_2(\tau)$ pull back under the Abel-Jacobi map $Y(\tau) \to
B(\tau)$ to the forms $\omega_1(\tau) = dw/z$ and $\omega_2(\tau) = w
\cdot dw/z$. The functions $a_k$ are modular for $\PSL(\ord_D \oplus
\ord_D^\vee)$ and the ratio of $a_0$ with the Igusa-Clebsch invariant
of weight two
\begin{equation}
\label{eqn:a0I2} a_0/I_2 \mbox{ where }  I_2 = -240a_0 + 40a_1a_5 - 16a_2a_4 + 6a_3^2 
\end{equation}
is $\PSL(\ord_D \oplus \ord_D^\vee)$-invariant.  Since $a_0(\tau)$ is
zero if and only if $\omega_2(\tau)$ has a double zero, $a_0/I_2$
covers an algebraic function on $X_D$ which vanishes along $W_D$.

To obtain an explicit model of $W_D$, we numerically sample $a_0/I_2$
using the model of $X_D$ in \cite{elkieskumar:hms} and the Eigenform
Location Algorithm we describe below.\footnote{Note that once we have
  an algebraic model, we will verify it rigorously without any
  reliance on floating-point computations.} Alternatively, one can
numerically sample $a_0/I_2$ using the functions in \texttt{Magma}
related to analytic Jacobians (cf. \cite{vwamelen:analyticjacobians}).
We then interpolate to find an exact rational function\footnote{The
  function $a_0/I_2$ is invariant under the involution
  $(\tau_1,\tau_2) \mapsto (\tau_2,\tau_1)$ which covers the deck
  transformation of the map $X_D$ onto its image in $\M_2$.  In the
  models in \cite{elkieskumar:hms}, this involution corresponds to the
  deck transformation of the map from $X_D$ to the $(r,s)$-plane.}
$w_D(r,s)/I_2(r,s)$ which equals $a_0/I_2$ in these models and whose
numerator appears in Table \ref{tab:wDrs}.

\paragraph{The Eigenform Location Algorithm (ELA).}  
To prove Theorems \ref{thm:wDrs} and \ref{thm:wDers}, in Section
\ref{sec:eformverify} we develop an Eigenform Location Algorithm (ELA,
Figure \ref{fig:ELA}). The algorithm takes as input a genus two curve whose Jacobian has real multiplication, and outputs the locations of the eigenforms. Recall that, for $Y\in \M_2$, there is a
natural pairing between $T_Y\M_2$ and the space of holomorphic
quadratic differentials $Q(Y)$ on $Y$.  There is a well-known formula
for this pairing which we recall in Section \ref{sec:algpairing} in
terms of a hyperelliptic model for $Y$.

Our location algorithm is based on the following theorem, which
is a consequence of Ahlfors's variational formula.
\begin{thm}
\label{thm:eformproducts}
For $\tau$ in the domain of the meromorphic function $Y : \hh \times
\hh \to \M_2$ defined by Equation \ref{eqn:Ywm}, the line in
$Q(Y(\tau))$ spanned by the quadratic differential
\[ q(\tau) = \omega_1(\tau) \cdot \omega_2(\tau) \]
annihilates the image of $(dY)_\tau$.
\end{thm}
\noindent Theorem \ref{thm:eformproducts} characterizes the eigenforms
$\omega_1(\tau)$ and $\omega_2(\tau)$ on $Y(\tau)$ up to permutation
and scale.  Using an algebraic model for the image of $X_D$ in $\M_2$
(as in, for example,
\cite{elkieskumar:hms,gruenewald:humbert,runge:endomorphism}) and the
formula in Section \ref{sec:algpairing}, we can use Theorem
\ref{thm:eformproducts} to identify eigenforms for real
multiplication.  This observation is the basis for ELA.  By running
ELA with floating point input, we numerically sample the function
$a_0/I_2$ defined in the previous paragraph and generate Tables
\ref{tab:wDrs} and \ref{tab:wD0rs} refered to in Theorems
\ref{thm:wDrs} and \ref{thm:wDers}.  By running ELA with input defined
over a function field $K$ over a number field
(e.g. $K=\qq(r)[s]/(w_D)$), we prove Theorems \ref{thm:wDrs} and
\ref{thm:wDers} using only rigorous arithmetic in $K$.

In \cite{km:correspondences}, we will describe a second method of
eigenform location based on explicit algebraic correspondences and
similar in spirit to \cite{vwamelen:examples,vwamelen:provingcm}. This
technique could be used to prove Theorem \ref{thm:wDrs} and such a proof would, unlike the proof in this
paper, be logically independent of \cite{elkieskumar:hms}.  We found
this correspondence method practical for certifying single eigenforms
and impractical for certifying positive dimensional families of
eigenforms.

\paragraph{Computer files.}  
Auxiliary files containing extra information on the Weierstrass curves
(omitted here for lack of space), as well as computer code to certify
our equations, are available from
\url{http://arxiv.org/abs/1406.7057}. To access these, download the
source file for the paper. This will produce both the \LaTeX \ file for
this paper and the computer code referenced below. The text file
\texttt{README} gives a reader's guide to the various auxiliary files.

\paragraph{Outline.}  
We conclude this Introduction by outlining the remaining sections of
this paper.
\begin{enumerate}
\item We begin in Section \ref{sec:jacRM} by studying families of
  marked Riemann surfaces whose Jacobians admit real multiplication.
  We prove that, for a Riemann surface $Y$ whose Jacobian has real
  multiplication, there is a symplectic basis $U$ for $H_1(Y,\rr)$
  consisting of eigenvectors for real multiplication (Proposition
  \ref{prop:eigenbasis}) and that the period matrix for $Y$ with
  respect to $U$ is diagonal (Proposition \ref{prop:diagperRM}).
  Using Ahlfors's variational formula, we deduce Proposition
  \ref{prop:eformproducts} which places a condition on eigenform
  products and generalizes Theorem \ref{thm:eformproducts}.
\item We then study the pairing between the vector spaces $Q(Y)$ and
  $T_Y\M_2$ for a genus two Riemann surface $Y$ birational to the
  plane curve defined by $z^2=f(w)$ with $\deg(f) = 5$.  There is a
  well known formula for this pairing in terms of the roots of $f(w)$.
  We recall this formula in Proposition \ref{prop:algpairingrts} and
  deduce Theorem \ref{thm:algpairingcoeffs} which gives a formula in
  terms of the coefficients of $f(w)$.
\item In Section \ref{sec:eformverify}, we combine the condition on
  eigenforms imposed by Proposition \ref{prop:eformproducts} with the
  pairing given in Section \ref{sec:algpairing} to give an Eigenform
  Location Algorithm.  We demonstrate our algorithm by identifying the
  eigenforms for real multiplication by $\ord_{12} = \zz\left[
    \sqrt{3} \right]$ on a particular genus two algebraic curve
  (Theorem \ref{thm:w12pt}).
\item In Section \ref{sec:wcurvecertify}, we implement ELA over
  function fields to certify our models of irreducible $W_D$ and prove
  Theorems \ref{thm:wDrs} and \ref{thm:gDxy}.
\item We then turn to reducible Weierstrass curves in Section
  \ref{sec:cuspspin}.  Using the technique in Section
  \ref{sec:wcurvecertify}, we can show that the curve $w_D^0(r,s) = 0$
  gives a birational model for {\em an} irreducible component of
  $W_D$.  In Section \ref{sec:cuspspin}, we explain how to distinguish
  between the irreducible components of $W_D$ by studying cusps,
  allowing us to prove Theorems \ref{thm:wDers} and \ref{thm:gDexy}.
\item In Section \ref{sec:arithmetic}, we discuss the proofs of the
  remaining theorems stated in this introduction concerning the
  arithmetic geometry of Weierstrass curves.
\end{enumerate}

\paragraph{Acknowledgments.}  
We thank Noam Elkies, Matt Emerton, Curt McMullen, Alex Wright and the
anonymous referees for helpful comments. A. Kumar was supported in
part by National Science Foundation grant DMS-0952486 and by a grant
from the MIT Solomon Buchsbaum Research Fund. R. E. Mukamel was
supported in part by National Science Foundation grant DMS-1103654.
We used the computer algebra systems \texttt{gp/Pari} \cite{pari},
\texttt{Magma} \cite{magma}, \texttt{Maple} \cite{maple} and
\texttt{Maxima} \cite{maxima}, extensively in our calculations. In
particular, most of the auxiliary computer files for verifying our
calculations are \texttt{Magma} files (however, they can be easily
adapted to different computer algebra systems, such as \texttt{Sage}
\cite{sage}).

\section{Jacobians with real multiplication}
\label{sec:jacRM}
Throughout this section, we fix the following:
\begin{itemize}
\item a compact topological surface $S$ of genus $g$,
\item an order $\ord$ in a totally real field $K$ of degree $g$ over
  $\qq$, and
\item a proper, self-adjoint embedding of rings $\rho: \ord
  \rightarrow \End(H_1(S,\zz))$.
\end{itemize}
Here, proper means that $\rho$ does not extend to a larger subring of
$K$ and self-adjoint is with respect to the intersection symplectic
form $E(S)$ on $H_1(S,\zz)$, i.e. for each $x,y \in H_1(S,\zz)$ and
$\alpha \in \ord$ we have $E(S)(\rho(\alpha) x,y) =
E(S)(x,\rho(\alpha) y)$.

Our goal for this section is to define and study the Teichm\"uller
space of the pair $(S,\rho)$.  The space $\teich(S,\rho)$ consists of
complex structures $Y$ on $S$ for which $\rho$ extends to real
multiplication by $\ord$ on $\jac(Y)$. In Proposition
\ref{prop:eigenbasis}, we show that there is a basis $U$ for
$H_1(S,\rr)$ consisting of eigenvectors for $\rho$.  In Proposition
\ref{prop:diagperRM}, we show that $Y$ is in $\teich(S,\rho)$ if and
only if the period matrix for $Y$ with respect to $U$ is diagonal.  In
Proposition \ref{prop:eformproducts}, we combine Ahlfors's variational
formula with Proposition \ref{prop:diagperRM} to derive a condition
satisfied by products of eigenforms for real multiplication on $Y \in
\teich(S,\rho)$.  The condition in Theorem \ref{thm:eformproducts}
follows easily from of Proposition \ref{prop:eformproducts}.

The results in the section are, for the most part, well known.  We
include them as background and to fix notation.  In Sections
\ref{sec:eformverify} and \ref{sec:wcurvecertify}, we will use
Proposition \ref{prop:eformproducts} to certify that certain algebraic
one-forms are eigenforms for real multiplication and show that the
equations in Table \ref{tab:wDrs} give algebraic models of Weierstrass
curves.  For additional background on abelian varieties, Jacobians and
their endomorphisms see \cite{birkenhakelange:cxabelianvarieties}, for
background on Hilbert modular varieties see
\cite{mcmullen:foliations,vdgeer:hms} and for background on
Teichm\"uller theory and moduli space of Riemann surfaces see
\cite{hubbard:teichthy,imayoshitaniguchi:teichthy,harris:moduli}.

\paragraph{Teichm\"uller space of $S$.}  
Let $\teich(S)$ be the Teichm\"uller space of $S$.  The space
$\teich(S)$ is the fine moduli space representing the functor sending
a complex manifold $B$ to the set of holomorphic families over $B$
whose fibers are marked by $S$ up to equivalence.  In particular, a
point $Y \in \teich(S)$ corresponds to an isomorphism class of Riemann
surface marked by $S$ and there are canonical isomorphisms $H_1(Y,\zz)
\cong H_1(S,\zz)$, $\pi_1(Y) \cong \pi_1(S)$, etc.  The space
$\teich(S)$ is a complex manifold homeomorphic to $\rr^{6g-6}$ and is
isomorphic to a bounded domain in $\cc^{3g-3}$.

\paragraph{Moduli space.}  
Let $\Mod(S)$ denote the {\em mapping class group} of $S$, i.e. the
group of orientation preserving homeomorphisms from $S$ to itself up
to homotopy.  The group $\Mod(S)$ acts properly discontinuously on
$\teich(S)$ and the quotient
\[ \M_g = \teich(S)/\Mod(S) \]
is a complex orbifold which coarsely solves the moduli problem for
unmarked families of Riemann surfaces homeomorphic to $S$.  We call
$\M_g$ the {\em moduli space of genus $g$ Riemann surfaces}.

\paragraph{Holomorphic one-forms and Jacobians.}  
For each $Y \in \M_g$, let $\Omega(Y)$ be the vector space of
holomorphic one-forms on $Y$ and let $\Omega(Y)^*$ be the vector space
dual to $\Omega(Y)$.  By complex analysis, $\dim_\cc \Omega(Y) = g$
and the map
\begin{equation}
\label{eqn:HinO}
f : H_1(Y,\rr) \rightarrow \Omega(Y)^* \mbox{ given by } f(a)(\omega) = \int_a \omega
\end{equation}
is an $\rr$-linear isomorphism.  In particular, $f(H_1(Y,\zz))$ is a
lattice in $\Omega(Y)^*$ and the quotient
\begin{equation}
\label{eqn:jac}
\jac(Y) = \Omega(Y)^* / f(H_1(Y,\zz))
\end{equation}
is a complex torus called the {\em Jacobian} of $Y$.  The Hermitian
form $H^*$ on $\Omega(Y)^*$ dual to the form
\begin{equation}
\label{eqn:defH}
H(\omega,\eta) =  \frac{i}{2} \int_Y \omega \wedge \overline{\eta} \mbox{ for each } \omega, \eta \in \Omega(Y)
\end{equation}
defines a principal polarization on $\jac(Y)$ since the pullback of
$\operatorname{Im}(H^*)$ under $f$ restricts to the intersection
pairing $E(Y)$ on $H_1(Y,\zz)$.

\paragraph{Jacobian endomorphisms.}  
An {\em endomorphism} of $\jac(Y)$ is a holomorphic homomorphism from
$\jac(Y)$ to itself.  Since $\jac(Y)$ is an abelian group, the
collection $\End(\jac(Y))$ of all endomorphisms of $\jac(Y)$ forms a
ring called the {\em endomorphism ring} of $\jac(Y)$.  Every
endomorphism $R \in \End(\jac(Y))$ arises from $\cc$-linear map
$\rho_a(R) : \Omega(Y)^* \to \Omega(Y)^*$ preserving the lattice
$f(H_1(Y,\zz))$.  The assignment
\begin{equation}
\label{eqn:anrep}
\rho_a : \End(\jac(Y)) \to \End(\Omega(Y)^*) \mbox{ given by } R \mapsto \rho_a(R)
\end{equation}
is an embedding of rings called the {\em analytic representation} of
$\End(\jac(Y))$.  We will denote by $\rho_a^*$ the representation of
$\End(\jac(Y))$ on $\Omega(Y)$ dual to $\rho_a$.  The assignment
\begin{equation}
\label{eqn:ratrep}
\rho_r : \End(\jac(Y)) \to \End(H_1(Y,\zz)) \mbox{ given by } \rho_r(R) = f^{-1} \circ \rho_a(R) \circ f
\end{equation}
is also an embedding of rings and is called the {\em rational
  representation} of $\End(\jac(Y))$.

For any endomorphism $R \in \End(\jac(Y))$, there is another
endomorphism $R^* \in \End(\jac(Y))$ called the {\em adjoint} of $R$
and characterized by the property that $\rho_r(R^*)$ is the
$E(Y)$-adjoint of $\rho_r(R)$.  The assignment $R \mapsto R^*$ defines
an (anti-)involution on $\End(\jac(Y))$ called the {\em Rosati
  involution}.

\paragraph{Real multiplication.}  
Recall that $K$ is a totally real number field of degree $g$ over
$\qq$ and $\ord$ is an order in $K$, i.e. a subring of $K$ which is
also a lattice.  We will say that $\jac(Y)$ {\em admits real
  multiplication by $\ord$} if there is
\begin{equation}
\label{eqn:defRM}
\mbox{a proper, self-adjoint embedding } \iota : \ord \rightarrow \End(\jac(Y)).
\end{equation}
Proper means that $\iota$ does not extend to a larger subring in $K$
and self-adjoint means that $\iota(\alpha)^* = \iota(\alpha)$ for each
$\alpha \in \ord$.  If $\ord$ is maximal (i.e. $\ord$ is not contained
in a strictly larger order in $K$) then an embedding $\ord
\to \End(\jac(Y))$ is automatically proper.

\paragraph{Teichm\"uller space of the pair $(S,\rho)$.}  
Recall that $\rho: \ord \to \End(H_1(S,\zz))$ is a proper and
self-adjoint embedding of rings.  For $Y \in \teich(S)$, we will say
that $\rho$ {\em extends to real multiplication by $\ord$ on
  $\jac(Y)$} if there is
\begin{equation}
\label{eqn:defextendrho}
\mbox{an embedding } \iota : \ord \to \End(\jac(Y)) \mbox{ satisfying } \rho_r \circ \iota = \rho.
\end{equation}
Equivalently, $\rho$ extends to real multiplication if and only if the
$\rr$-linear extension of $f \circ \rho(\alpha) \circ f^{-1}$ to
$\Omega(Y)^*$ is $\cc$-linear for each $\alpha \in \ord$.  Since
$\rho$ is proper and self-adjoint, an $\iota$ as in Equation
\ref{eqn:defextendrho} is automatically proper and self-adjoint in the
sense of the previous paragraph.  In Equation \ref{eqn:defextendrho},
we have implicitly identified $H_1(Y,\rr)$ with $H_1(S,\rr)$ via the
marking.

We define the Teichm\"uller space of the pair $(S,\rho)$ to be the space
\begin{equation}
\label{eqn:defTeichpair}
\teich(S,\rho) = \left\{ Y \in \teich(S) : \mbox{$\rho$ extends to
  real multiplication by $\ord$ on $\jac(Y)$} \right\}.
\end{equation}
If $\rho_1$ and $\rho_2$ are two proper, self-adjoint embeddings $\ord
\to \End(H_1(S,\zz))$ and $g \in \Mod(S)$ is a mapping class such that
the induced map on homology $g_* \in \End(H_1(S,\zz))$ conjugates
$\rho_1(\alpha)$ to $\rho_2(\alpha)$ for each $\alpha \in \ord$, then
$g$ gives a biholomorphic map between $\teich(S,\rho_1)$ and
$\teich(S,\rho_2)$.

\paragraph{Symplectic $K$-modules and their eigenbases.}  
The representation
\[ \rho_K = \rho \otimes_\zz \qq : K \rightarrow \End(H_1(S,\qq)) \]
turns $H_1(S,\qq)$ into a $K$-module.  We begin our study of
$\teich(S,\rho)$ by showing that there is a unique symplectic
$K$-module that arises in this way.

Let $E(\Tr)$ be the symplectic {\em trace form} on $K \oplus K$ defined by
\begin{equation}
\label{eqn:trpairing}
E(\Tr)\left( (x_1,y_1),(x_2,y_2) \right) = \Tr^K_\qq (x_1 y_2 - y_1 x_2).
\end{equation}
It is easy to check that multiplication by $k \in K$ is self-adjoint
for $E(\Tr)$.
\begin{prop}
\label{prop:sympKmods}
Regarding $H_1(S,\qq)$ as a $K$-module via $\rho_K = \rho \otimes_\zz
\qq$, there is a $K$-linear isomorphism
\[ T: K \oplus K \rightarrow H_1(S,\qq) \]
which is symplectic for the trace from $E(\Tr)$ on $K \oplus K$ and
the intersection form $E(S)$ on $H_1(S,\qq)$.
\end{prop}
\begin{proof}
Choose any $x \in H_1(S,\qq)$ and set $L = \rho_K(K) \cdot x$.  Since
$\rho_K$ is self-adjoint, $L$ is isotropic.  The non-degeneracy of the
intersection form $E(S)$ ensures that there is a $y \in H_1(S,\qq)$
such that $E(S)\left( \rho_K(k) \cdot x, y \right)= \Tr^K_\qq(k)$.
Define a map $T : K \oplus K \rightarrow H_1(S,\qq)$ by the formula
\[ T(k_1,k_2) = \rho_K(k_1) \cdot x + \rho_K(k_2) \cdot y. \]
Clearly, the map $T$ is $K$-linear.  An easy computation shows that
$T$ satisfies $E(\Tr)(v,w) = E(S)(T(v),T(w))$ for each $v,w \in K
\oplus K$ which, together with the non-degeneracy of $E(\Tr)$, implies
that $T$ is a symplectic vector space isomorphism.
\end{proof}
Now let $h_1,\dots,h_g : K \rightarrow \rr$ be the $g$ places for $K$.
Proposition \ref{prop:sympKmods} allows us to show that there is a
symplectic basis for $H_1(S,\rr)$ adapted to $\rho$.
\begin{prop}
\label{prop:eigenbasis}
There is a symplectic basis $U = \left<
a_1,\dots,a_g,b_1,\dots,b_g\right>$ for $H_1(S,\rr)$ such that
\begin{equation}
\label{eqn:eigenbasis}
\rho(\alpha) a_i = h_i(\alpha) \cdot a_i \mbox{ and } \rho(\alpha) b_i = h_i(\alpha) \cdot b_i \mbox{ for each $\alpha \in \ord$.}
\end{equation}
\end{prop}
\begin{proof}
Since the group $H_1(S,\qq)$ is isomomorphic as a symplectic
$K$-module to $K \oplus K$ with the trace pairing $E(\Tr)$
(Proposition \ref{prop:sympKmods}), it suffices to construct an
analogous basis for $(K \oplus K) \otimes_\qq \rr$.  Let
$\alpha_1,\dots,\alpha_g$ be an arbitrary $\qq$-basis for $K$.  Since
$\Tr : K \times K \to \qq$ is non-degenerate, we can choose
$\beta_1,\dots,\beta_g \in K$ so that $\Tr^K_\qq(\alpha_i \beta_j) =
\delta_{ij}$.  Setting
\[ a_i = \sum_{j = 1}^g (\alpha_j,0) \otimes h_i(\beta_j)  \mbox{ and } b_i = \sum_{j=1}^g (0,\beta_j) \otimes h_i(\alpha_j) \]
yields a basis with the desired properties.
\end{proof}

\paragraph{The period map.} 
Now let $\mathcal H_g$ be the Siegel upper half-space consisting of $g
\times g$ symmetric matrices with positive definite imaginary part.
The space $\mathcal H_g$ is equal to an open, bounded and symmetric
domain in the $(g^2+g)/2$-dimensional space of all symmetric matrices.

As we now describe, the basis $U$ for $H_1(S,\rr)$ given by
Proposition \ref{prop:eigenbasis} allows us to define a holomorphic
period map from $\teich(S)$ to $\mathcal H_g$.  For $Y \in \teich(S)$,
we can view $U$ as a basis for $H_1(Y,\rr)$ via the marking by $S$.
Let $\left< \omega_1(Y),\dots,\omega_g(Y) \right>$ be the basis for
$\Omega(Y)$ dual to $U$, i.e. such that $\int_{a_j} \omega_k(Y) =
\delta_{jk}$.  The {\em period map} is defined by
\begin{equation}
\label{eqn:defpermap}
P : \teich(S) \to \mathcal H_g \mbox{ where } P_{jk}(Y) = \int_{b_j} \omega_k(Y).
\end{equation}
Our next proposition characterizes the points in $\teich(S,\rho)$ \cite[\S 6]{mcmullen:billiards}.
\begin{prop}
\label{prop:diagperRM}
For $Y \in \teich(S)$, the homomorphism $\rho$ extends to real
multiplication by $\ord$ on $\jac(Y)$ if and only if the period matrix
$P(Y)$ is diagonal.
\end{prop}
\begin{proof}
From $\int_{a_j} \omega_k(Y) = \delta_{jk}$ and $P_{jk}(Y) = P_{kj}(Y)
= \int_{b_j} \omega_k(Y)$ we see that the map $f : H_1(Y,\rr) \to
\Omega(Y)^*$ of Equation \ref{eqn:HinO} satisfies $f(b_j) =
\sum_{k=1}^g P_{jk}(Y) \cdot f(a_k)$.  In matrix--vector notation, we
have
\begin{equation}
\label{eqn:Pandf}
(f(b_1),f(b_2),\dots,f(b_g)) = P(Y) \cdot (f(a_1),f(a_2),\dots,f(a_g)).
\end{equation}
For $\alpha \in \ord$, let $h(\alpha)$ be the $g \times g$ diagonal
matrix with diagonal entries $(h_1(\alpha),\dots,h_g(\alpha))$.  From
Equation \ref{eqn:eigenbasis}, the map $T(\alpha) = f \circ
\rho(\alpha) \circ f^{-1}$ extends $\cc$-linearly to $\Omega(Y)^*$ if
and only if the matrix for $T(\alpha)$ is $h(\alpha)$ with respect to
both the basis $\left<f(a_1),\dots,f(a_g) \right>$ and the basis
$\left< f(b_1),\dots,f(b_g) \right>$.  From Equation \ref{eqn:Pandf}
this happens if and only if $P(Y)$ commutes with $h(\alpha)$.  Since
the embeddings $h_1,\dots,h_g : \ord \to \rr$ are pairwise distinct,
$P(Y)$ commutes with $h(\alpha)$ for every $\alpha \in \ord$ if and
only if $P(Y)$ is diagonal.
\end{proof}

\paragraph{Eigenforms for real multiplication.}  
For $Y \in \teich(S,\rho)$ and $\iota$ satisfying $\rho_r \circ \iota
= \rho$, we saw in the proof of Proposition \ref{prop:diagperRM} that
the matrix for $\rho_a(\iota(\alpha))$ with respect to the basis
$\left<f(a_1),\dots,f(a_g)\right>$ for $\Omega(Y)^*$ is the diagonal
matrix $h(\alpha)$.  Since this basis is dual to the basis
$\left<\omega_1(Y),\dots,\omega_g(Y)\right>$ for $\Omega(Y)$, we see
that $\rho_a^*(\iota(\alpha)) \in \End(\Omega(Y))$ stabilizes
$\omega_i(Y)$ up to scale.  We record this fact in the following
proposition.
\begin{prop}
\label{prop:eigenforms}
For $Y \in \teich(S,\rho)$ and $\alpha \in \ord$, we have that $\rho_a^*(\alpha) \omega_i(Y) = h_i(\alpha) \omega_i(Y)$.
\end{prop}
\noindent In light of Proposition \ref{prop:eigenforms}, we call the
non-zero scalar multiples of $\omega_i(Y)$ the {\em $h_i$-eigenforms
  for $\ord$}.

\paragraph{Moduli of abelian varieties.}  
Now consider the homomorphism $M : \PSp(H_1(S,\rr)) \to
\PSp_{2g}(\rr)$ sending a projective symplectic automorphism of
$H_1(S,\rr)$ to its matrix with respect to $U$.  There is an action of
$\PSp_{2g}(\rr)$ on $\mathcal H_g$ by holomorphic automorphisms via
generalized M\"obius transformations such that, for $h \in \Mod(S)$
inducing $h_* \in \End(H_1(S,\zz))$, we have
\begin{equation}
\label{eqn:modperiods}
M(h_*) \cdot P(Y) = P(h \cdot Y).
\end{equation}
We conclude that the period map $P : \teich(S) \to \mathcal H_g$
covers a holomorphic map
\begin{equation}
\label{eqn:MgtoAg} \jac : \M_g \to A_g = \mathcal H_g / \Gamma_\zz \mbox{ where } \Gamma_\zz = M(\PSp(H_1(S,\zz))).
\end{equation}
We also call this map the period map and denote it by $\jac$ since the
space $A_g$ has a natural interpretation as a moduli space of
principally polarized abelian varieties so that $\jac$ is simply the
map sending a Riemann surface to its Jacobian.

\paragraph{Hilbert modular varieties.}  
Let $\Delta_g$ denote the collection of diagonal matrices in $\mathcal
H_g$ and let $\PSp(H_1(S,\zz),\rho)$ denote the subgroup of
$\PSp(H_1(S,\zz))$ represented by symplectic automorphisms commuting
with $\rho(\alpha)$ for each $\alpha \in \ord$.  The group
$\Gamma_\rho = M(\PSp(H_1(S,\zz),\rho))$ consists of matrices whose $g
\times g$ blocks are diagonal.  Consequently, $\Gamma_\rho$ preserves
$\Delta_g$ and, by Proposition \ref{prop:diagperRM}, the map
$\teich(S,\rho) \to A_g$ covered by the period map $P$ factors through
the orbifold
\begin{equation}
\label{eqn:defHMV}
X_\rho = \Delta_g / \Gamma_\rho.
\end{equation}
The space $X_\rho$ has a natural interpretation as a moduli space of
abelian varieties with real multiplication.  Each of the complex
orbifolds $\M_g$, $A_g$ and $X_\rho$ can be given the structure of an
algebraic variety so that the map in the period map $\jac$ and the map
$X_\rho \to A_g$ covered by the inclusion $\Delta_g \to \mathcal H_g$
are algebraic.  The variety $X_\rho$ is called a {\em Hilbert modular
  variety}.

\paragraph{Tangent and cotangent space to $\teich(S)$.}  
For $Y \in \teich(S)$, let $B(Y)$ denote the vector space of
$L^\infty$-Beltrami differentials on $Y$.  The measurable Riemann
mapping theorem can be used to give a marked family over the unit ball
$B^1(Y)$ in $B(Y)$ and construct a holomorphic surjection $\phi:
B^1(Y) \to \teich(S)$ with $\phi(0) = Y$.  There is a pairing between
$B(Y)$ and the space of holomorphic quadratic differentials $Q(Y)$ on
$Y$ given by
\begin{equation}
\label{eqn:BYQYpair}
B(Y) \times Q(Y) \to \cc \mbox{ where }(\mu,q) \mapsto \int_Y \mu \cdot q.
\end{equation}
Now let $Q(Y)^\perp \subset B(Y)$ be the vector subspace consisting of
Beltrami differentials annihilating every quadratic differential under
the pairing in Equation \ref{eqn:BYQYpair}.  By Teichm\"uller theory,
the space $Q(Y)^\perp$ is closed, has finite codimension and is equal
to the kernel of $d\phi_0$.  The tangent space $T_Y \teich(S)$ is
isomorphic to $B(Y) / Q(Y)^\perp$ and the pairing in Equation
\ref{eqn:BYQYpair} covers a pairing between $T_Y \teich(S)$ and $Q(Y)$
giving an isomorphism
\begin{equation}
\label{eqn:coTYQY}
T_Y^* \teich(S) \cong Q(Y).
\end{equation}
The pairing in Equation \ref{eqn:BYQYpair} and the isomorphism in
Equation \ref{eqn:coTYQY} are $\Mod(S)$-equivariant, and they give
rise to a pairing between $Q(Y)$ and the orbifold tangent space $T_Y
\M_g$.

\paragraph{Ahlfors's variational formula and eigenform products.}  
Ahlfors's variational formula expresses the derivative of the period
map in terms of quadratic differentials.
\begin{thm}[Ahlfors, \cite{ahlfors:complexanalytic}] 
For any $Y \in \teich(S)$, the derivative of the
$(j,k)^{\text{\tiny{th}}}$ coefficient of the period map is the
quadratic differential
\begin{equation}
\label{eqn:ahlforsvf}
d\left( P_{jk} \right)_Y = \omega_j(Y) \cdot \omega_k(Y).
\end{equation}
\end{thm}
\noindent Combining Equation \ref{eqn:ahlforsvf} with our
characterization of the period matrices of $Y \in \teich(S,\rho)$
yields the following proposition, implicit in \cite[proof of Theorem 7.5]{mcmullen:billiards}. 
\begin{prop}
\label{prop:eformproducts}
Suppose $B$ is a smooth manifold and $g : B \rightarrow
\teich(S,\rho)$ is a smooth map.  For each $j \neq k$ and $b \in B$,
the quadratic differential
\[ q_{jk}(b) = \omega_j(g(b)) \cdot \omega_k(g(b)) \in Q(g(b)) \]
annihilates the image of $dg_{b}$ in $T_{g(b)}\teich(S)$.
\end{prop}
\begin{proof}
In light of Proposition \ref{prop:diagperRM}, the image of $P \circ g$
is contained within the set $\Delta_g$ of diagonal matrices in
$\mathcal H_g$.  For $j \neq k$, the composition $P_{jk} \circ g : B
\rightarrow \cc$ is identically zero.  The differential $dP_{jk}$
annihilates the image of $dg_{b}$ by the chain rule and is equal to
$q_{jk}(b)$ by Ahlfors's variational formula.
\end{proof}
\noindent Theorem \ref{thm:eformproducts} is a special case of
Proposition \ref{prop:eformproducts}.
\begin{proof}[Proof of Theorem \ref{thm:eformproducts}]
Fix $\tau$ in the domain for the map $Y : \hh \times \hh \to \M_2$
defined in Equation \ref{eqn:Ywm} and let $B$ be a neighborhood of
$\tau$ on which $Y$ lifts to a map $\widetilde Y : B \to \teich(S)$
for a genus two surface $S$.  Identify $H_1(S,\zz)$ with
$H_1(\widetilde{Y}(\tau),\zz)$ via the marking and with the lattice
$\Pi(\tau) \cdot \zz^4 = H_1(B(\tau),\zz)$ (Equation \ref{eqn:Btau})
via the Abel-Jacobi map $\widetilde{Y}(\tau) \to B(\tau)$ and let
$\rho : \ord_D \to \End(H_1(S,\zz))$ be the proper, self-adjoint
embedding with $\rho(\alpha)$ equal to multiplication by the diagonal
matrix $\operatorname{Diag}\left( h_1(\alpha),h_2(\alpha) \right)$ on
$\Pi(\tau) \cdot \zz^4$.  Clearly, the lift $\widetilde Y$ maps $B$
into $\teich(S,\rho)$ and Proposition \ref{prop:eformproducts} shows
that the product $\omega_1(\tau) \cdot \omega_2(\tau)$ annihilates the
image of $d\widetilde{Y}_\tau$.
\end{proof}

\paragraph{Genus two Jacobians with real multiplication.}  
For typical pairs $(S,\rho)$, we know little else about the space
$\teich(S,\rho)$ including whether or not $\teich(S,\rho)$ is empty.
For the remainder of this section, we impose the additional assumption
that $g = 2$ so that we can say more.

Let $D$ be the discriminant of $\ord$.  The first special feature when
$g = 2$ is that the order $\ord$ is determined by $D$ and is
isomorphic to $\ord_D = \zz\left[ \frac{D+\sqrt{D}}{2} \right]$.  The
discriminants of real quadratic orders are precisely the integers
$D>0$ and congruent to $0$ or $1 \bmod 4$, and the order of
discriminant $D$ is an order in a real quadratic field if and only if
$D$ is not a square.\footnote{Rings with square discriminants
  correspond to orders in $\zz \times \zz$ and can in principle be
  treated similarly to those we consider in this paper.  Since
  equations for $X_{d^2}$ do not appear in \cite{elkieskumar:hms}, we
  will not consider such rings in this paper.} The discriminant $D$ is
{\em fundamental} and the order $\ord_D$ is {\em maximal} if $\ord_D$
is not contained in a larger order in $\ord_D \otimes \qq$.

The second special feature when $g = 2$ is that, for each real
quadratic order $\ord$, there is a unique proper, self-adjoint
embedding $\rho : \ord \to H_1(S,\zz)$ up to conjugation by elements
of $\Sp(H_1(S,\zz))$ (\cite{runge:endomorphism}, Theorem 2).  Since
$\Mod(S) \to \Sp(H_1(S,\zz))$ is onto, the spaces $\teich(S,\rho)$ and
$X_\rho$ and the maps to $\teich(S,\rho) \to \M_2$ and $X_\rho \to
A_2$ are determined by $D$ up to isomorphism.  The Hilbert modular
variety $X_\rho$ is isomorphic to the {\em Hilbert modular surface of
  discriminant $D$}
\begin{equation}
\label{eqn:XD}
\begin{array}{c}
X_D = \hh \times \hh / \PSL(\ord_D \oplus \ord_D^\vee)\mbox{ where } \ord_D^\vee = \frac{1}{\sqrt{D}} \cdot \ord_D \mbox{ and} \\
 \PSL(\ord_D \oplus \ord_D^\vee) = \left\{ \begin{pmatrix} a & b \\ c & d \end{pmatrix} \in \PSL_2(K) : \begin{array}{c} ad -bc = 1, a,d \in \ord_D \\ b \cdot \ord_D^\vee \subset \ord_D \mbox{ and } c \cdot \ord_D \subset \ord_D^\vee \end{array} \right\}.
\end{array}
\end{equation}
The two places $h_1, h_2 : \ord_D \to \rr$ give two homomorphisms
$\PSL(\ord_D \oplus \ord_D^\vee) \to \PSL_2(\rr)$ which we also denote
by $h_1$ and $h_2$.  The action of $M \in
\PSL(\ord_D\oplus\ord_D^\vee)$ on $\hh \times \hh$ is the ordinary
action of $h_1(M)$ by M\"obius transformation the first coordinate and
by $h_2(M)$ on the second.

The third special feature when $g=2$ is that the algebraic map $\jac :
\M_2 \to A_2$ is birational with rational inverse $\jac^{-1}$.
Composing $\jac^{-1}$ with the natural map $X_D \to A_2$ gives a
rational inverse period map
\begin{equation}
\label{eqn:invjac}
 \jac_D^{-1} : X_D \to \M_2
\end{equation}
whose image is covered by $\teich(S,\rho)$.  

\section{Quadratic differentials and residues}
\label{sec:algpairing}
To make use of Proposition \ref{prop:eformproducts} for a complex
structure $Y \in \teich(S,\rho)$ represented by an algebraic curve, we
will need an algebraic formula for the pairing between $Q(Y)$ and
$T_Y\teich(S)$ described in Equation \ref{eqn:coTYQY}.  When the genus
of $S$ is two, the curve $Y$ is biholomorphic to a hyperelliptic curve
defined by an equation of the form $z^2 = f_a(w)$ where
\begin{equation}
\label{eqn:fa}
f_a(w) = w^5+a_4 w^4+a_3 w^3+a_2 w^2 + a_1 w + a_0 
\end{equation}
for some $a = (a_0,\dots,a_4) \in \cc^5$.  There is a well known
formula for the pairing between $Q(Y)$ and $T_Y\teich(S)$ for such
curves involving residues and the roots of $f_a(w)$.  We recall this
formula in Proposition \ref{prop:algpairingrts}.  From Proposition
\ref{prop:algpairingrts}, we deduce Theorem \ref{thm:algpairingcoeffs}
which gives a formula in terms of the coefficients of $f_a(w)$.
Theorem \ref{thm:algpairingcoeffs} is useful from a computational
standpoint since the field $\qq(a_0,\dots,a_4)$ is typically simpler
than the splitting field of $f_a(w)$ over $\qq$.

\paragraph{Pairing and polynomial coefficients.} 
For any $a=(a_0,\dots,a_4) \in \cc^5$, let $f_a(w)$ be the polynomial
defined in Equation \ref{eqn:fa} and set
\begin{equation}
\label{eqn:V}
V = \left\{ a \in \cc^5 : \Disc(f_a(w)) \neq 0\right\}.
\end{equation}
Consider the holomorphic and algebraic map
\[ Y : V \to \M_2 \mbox{ where $Y(a)$ is birational to the curve defined by $z^2 = f_a(w)$.} \]
While the universal curve over $\M_2$ is not a fiber bundle, its
pullback to $V$ is a fiber bundle. By Teichm\"uller theory, the map
$Y$ lifts locally to $\teich(S)$ and, for any $a \in V$, the
derivative $(dY)_a$ gives rise to a pairing
\begin{equation}
\label{eqn:QYTaV}
 Q(Y(a)) \times T_aV \to \cc.
\end{equation}
The tangent space $T_aV$ is naturally isomorphic to $\cc^5$ since $V$
is open in $\cc^5$ and we can identify the vector space $Q(Y(a))$ with
$\cc^3$ by associating $x =(x_0,x_1,x_2)$ with the quadratic
differential
\begin{equation}
\label{eqn:qx}
q_x = Q_x(w) \cdot \frac{dw^2}{f_{a}(w)} \mbox{ where $Q_x(w) = x_0+x_1w+x_2w^2$}.
\end{equation}
Our main goal for this section is to establish a formula for the
pairing in Equation \ref{eqn:QYTaV} in these coordinates.  We start
with the following proposition.
\begin{prop}
For any integer $k \geq 2$, the function
\begin{equation}
\label{eqn:Macoeff}
M_k(a) = \sum_{\substack{r \in Z(f_a)}} \frac{r^{k}}{\left( r f_a'(r) \right)^2} \mbox{ where } Z(f_a) = \left\{ r \in \cc : f_a(r) = 0 \right\}
\end{equation}
is a rational function in $\cc\left( a_0,\dots,a_4 \right)$ and the
product $\Disc\left( f_a(w) \right) \cdot M_k(a)$ is in
$\cc[a_0,\dots,a_4]$.
\end{prop}
\begin{proof} 
The function $M_k(a)$ is a rational function in $\cc\left(
a_0,\dots,a_4 \right)$ since the sum in Equation \ref{eqn:Macoeff} is
a symmetric rational function of the roots of $f_a(w)$.  That
$\Disc\left( f_a(w) \right) \cdot M_k(a)$ is a polynomial for $k \geq
2$ follows from the fact that $\Disc(f_a(w)) / f_a'(r)^2 = \Disc\left(
f_a(w)/(w-r) \right)$ for each $ r \in Z(f_a)$.
\end{proof}
\noindent We can now state our main theorem for this section.
\begin{thm}
\label{thm:algpairingcoeffs}
For any $a \in V$, $v \in T_aV$ and $x \in \cc^3$, the pairing between
the quadratic differential $q_x \in Q(Y(a))$ and $v$ is given by
\begin{equation} 
\label{eqn:algpairingcoeffs} q_x(v) = (2 \pi) \cdot x^T \cdot M(a) \cdot v
\end{equation}
where $M(a)$ is the $3 \times 5$-matrix with entries in
$\Disc(f_a(w))^{-1} \cdot \cc\left[ a_0,\dots,a_4 \right]$ whose
$(j,k)^{\text{\tiny{th}}}$ coefficient is the rational function
$(M(a))_{jk} = M_{j+k}(a)$ defined in Equation \ref{eqn:Macoeff}.
\end{thm}
\noindent Many computer algebra systems will readily compute an
explicit formula for $M_k(a)$ for any particular $k$ and in the
auxiliary files we include an explicit formula for $M(a)$ (which
involves the polynomials $M_k(a)$ for $2 \leq k \leq 8$).  We will
prove Theorem \ref{thm:algpairingcoeffs} at the end of this section.

\paragraph{Pairing and roots of polynomials.}  
For any $r = (r_1,\dots,r_5) \in \cc^5$, let $a(r) =
(a_0(r),\dots,a_4(r))$ be the coefficients of the monic, degree five
polynomials with roots $r_i$ and define
\[ V^{rts} = \left\{ r \in \cc^5 : a(r) \in V \right\}. \]
The universal curve over $\M_2$ pulls back under $Y \circ a : V^{rts}
\to \M_2$ to a fiber bundle.  As in the previous paragraph,
Teichm\"uller theory gives rise to a natural pairing between
$T_rV^{rts}$ and $Q(Y(a(r)))$ for any $r \in V^{rts}$.  The tangent
space $T_rV^{rts}$ is naturally isomorphic to $\cc^5$ since $V^{rts}$
is open in $\cc^5$ and we can identify $Q(Y(a(r)))$ with $\cc^3$ via
Equation \ref{eqn:qx}.

Let $N(r)$ be the $3 \times 5$-matrix whose $(j,k)$-th coefficient is
the rational function
\begin{equation}
\label{eqn:Nr}
N(r)_{jk} = -\operatorname{Res}_{r_j} \left( \frac{w^{k-1}}{f_{a(r)}(w)} \right) = -r_j^{k-1} \cdot \prod_{m \neq j} \frac{1}{r_m-r_j} \mbox{ for $1 \leq j \leq 5$ and $1 \leq k \leq 3$.}
\end{equation}
Our next proposition expresses the pairing $Q(Y(a(r)))\times T_r
V^{rts} \to \cc$ in terms of $N(r)$.
\begin{prop}
\label{prop:algpairingrts}
For any $r \in V$, $u \in T_rV^{rts}$ and $x \in \cc^3$, the pairing
between the quadaratic differential $q_x \in Q(Y(a(r)))$ and $u$ is
given by
\begin{equation}
\label{eqn:algpairingrts}
q_x(u) = (2 \pi) \cdot x^T \cdot N(r) \cdot u.
\end{equation}
\end{prop}
\noindent Proposition \ref{prop:algpairingrts} is
well-known.\footnote{Proposition \ref{prop:algpairingrts} can be
  established from the general discussion of deformations of Riemann
  surfaces in \cite{mcmullen:navigating}; nearly identical statements
  appear in \cite{pilgrim:coursenotes} pg. 50 and
  \cite{hubbardschleicher:spider} Proposition 7.3.  Equation
  \ref{eqn:algpairingrts} differs from the equations in
  \cite{pilgrim:coursenotes} and \cite{hubbardschleicher:spider} by a
  constant factor arising from different definitions and the fact that
  our formula is in genus two.}  We include a proof for completeness.
\begin{proof}
Recall that the Riemann surface $Y(a)$ is birational to the algebraic
curve $z^2 = f_a(w)$.  Let $r = (r_1,\dots,r_5)$ and $u=
(u_1,\dots,u_5)$.  One can construct a smoothly varying family of
diffeomorphisms
\[ \Phi_t(w,z) = (W_t(w,z),Z_t(w,z)) : Y\left( a(r) \right) \rightarrow Y\left( a(r+tu) \right) \mbox{ for $t$ small} \]
such that $W_t(w,z) = w+t u_i$ for $w$ in a neighborhood of
$w^{-1}(r_j)$, $W_t(w,z) = w$ for $w$ in a neighborhood of $\infty$
and $\Phi_0 = \operatorname{id}|_{ Y(a(r))}$.  Since $\Phi_t$ is
holomorphic for large $w$, the Beltrami differential $\mu(\Phi_t)$ is
supported away from $w^{-1}(\infty)$ and our computation below is
unaffected by our identification of $Y(a(r+tu))$ with the affine plane
curve $z^2 = f_{a(r+tu)}(w)$.

The family of Beltrami differentials $\mu\left(\Phi_t\right) =
\overline{\partial} \Phi_t/\partial \Phi_t$ provides a lift of the
local map from $V^{rts}$ into $\teich(S)$ to a map into the unit ball
$B^1(Y(a(r)))$ in the space of $L^\infty$-Beltrami differentials on
$Y(a(r))$.  To compute $q_x(u)$, we compute $\mu\left( \Phi_t \right)$
to first order in $t$ and evaluate the right hand side of
\begin{equation}
\label{eqn:qxu} q_x(u) = \lim_{t \rightarrow 0} \frac{1}{t} \int_{Y\left( a(r) \right)} \mu(\Phi_t) q_x.
\end{equation}
Compare Equation \ref{eqn:qxu} with Equation \ref{eqn:BYQYpair}.
Since $\Phi_t$ is holomorphic in a neighborhood of the zeros and poles
of the meromorphic one-form $dw$ on $Y\left( a(r) \right)$, the
Beltrami differential $\mu\left( \Phi_t \right)$ satisfies
\begin{equation}
\label{eqn:muphit}
\mu\left( \Phi_t \right) = \frac{\overline{\partial} \Phi_t}{\partial \Phi_t} = \frac{\partial W_t(w,z) / \partial \overline{w} \cdot d\overline{w}}{\partial W_t(w,z) / \partial w \cdot dw} = \frac{\partial W_t(w,z)/\partial \overline{w}}{1+O(t)} \cdot \frac{d\overline{w}}{dw} = \partial W_t(w,z)/\partial \overline{w} \cdot \frac{d\overline{w}}{dw} + O(t^2).
\end{equation}
The product $q_x \cdot \mu\left( \Phi_t \right)$ is supported away
from small disks about the Weierstrass points of
$Y\left(a(r)\right)$. From Equation \ref{eqn:muphit} we see that the
product $q_x \cdot \mu\left( \Phi_t \right)$ is nearly exact in such a
neighborhood, i.e.
\begin{equation}
 q_x \cdot \mu\left( \Phi_t \right) = \frac{Q_x(w) \partial W_t / \partial \overline{w} }{f_{a(r)}(w)}|dw|^2 + O(t^2) = d \eta +O(t^2) \mbox{ where } \eta = \frac{i}{2} \cdot \frac{(W_t(w,z)-w) \cdot Q_x(w)}{f_{a(r)}(w)} dw.
\end{equation}
The factor of $i/2$ arises from the equation $|dw|^2 = (i/2) \cdot dw
\wedge d\overline{w}$.  Stokes' theorem gives
\begin{equation}
 \int_{Y(a(r))} q_x \cdot \mu\left( \Phi_t \right) = \int_{C_\infty} \eta + \sum_{j = 1}^5 \int_{C_j} \eta + O(t^2)
\end{equation}
where $C_j$ is a small loop around $w^{-1}(r_j)$ and $C_\infty$ is a
small loop around $w^{-1}(\infty)$.  Since the one-form $\eta$ equals
$(i/2) \cdot t u_j Q_x(w) dw / f_{a(r)}(w)$ in a neighborhood of
$w^{-1}(r_j)$ and is identically zero in a neighborhood of
$w^{-1}(\infty)$, we conclude that
\[ q_x (u) = (-2 \pi)\cdot \sum_{j=1}^5 u_j \operatorname{Res}_{r_j} \left( \frac{Q_x(w)}{f_{a(r)}(w)} dw \right) = (2 \pi) \cdot x^T \cdot N(r) \cdot u. \]
Note the extra factor of two arising from the fact that $w$ maps $C_j
\subset Y\left( a(r) \right)$ to a closed curve on the $w$-sphere
winding twice about $r_j$.
\end{proof}
We are now ready to prove Theorem \ref{thm:algpairingcoeffs}.
\begin{proof}[Proof of Theorem \ref{thm:algpairingcoeffs}]
For any $r \in V^{rts}$, set $Q(r) = Q(Y(a(r)))$.  The pairings $Q(r)
\times T_{a(r)}V \to \cc$ and $Q(r) \times T_r V^{rts} \to \cc$
correspond to linear maps $L_r : T_rV^{rts} \to (Q(r))^*$ and $L_a :
T_{a(r)} V \to (Q(r))^*$ which are related to one another by
composition with the derivative of $a : V^{rts} \to V$, i.e. $L_r =
L_a \circ da_r$.  By Proposition \ref{prop:algpairingrts}, the matrix
$N(r)$ is the matrix for $L_r$ with respect to the bases for $Q(r)$
and $T_rV^{rts}$ discussed above.

One can give a conceptual proof that the derivative $da_r$, the matrix
$N(r)$ defined by Equation \ref{eqn:Nr} and the matrix $M(a)$ whose
$(j,k)$-th coefficient is the polynomial $M_{j+k}(a)$ defined in
Equation \ref{eqn:Macoeff} are related by
\begin{equation}
\label{eqn:NrMada}
N(r) = \frac{M(a(r)) \cdot da_r}{\Disc(f_{a(r)}(w))}. 
\end{equation}
We include a program to verify Equation \ref{eqn:NrMada} in the
auxiliary computer files.  We conclude $q_x(v)$ and $M(a)$ satisfy
Equation \ref{eqn:algpairingcoeffs}.
\end{proof}

\section{The eigenform location algorithm}
\label{sec:eformverify}
In this section, we develop a method of eigenform location based on
the condition on eigenform products imposed by Proposition
\ref{prop:eformproducts} and the formula in Theorem
\ref{thm:algpairingcoeffs}.  We demonstrate our method by proving the
following theorem.
\begin{thm}
\label{thm:w12pt}
The Jacobian of the algebraic curve $Y$ with Weierstrass model
\begin{equation} \label{eqn:w12pt} z^2 = w^5-2 w^4 - 12 w^3-8 w^2+52 w +24 \end{equation}
admits real multiplication by $\ord_{12} = \zz\left[ \sqrt{3} \right]$
with eigenforms $dw/z$ and $w \cdot dw/z$.
\end{thm}
\noindent Since the one-form $dw/z$ has a double zero, Theorem
\ref{thm:w12pt} immediately implies the following.
\begin{cor}
The one-form up to scale $(Y,[dw/z])$ defined by Equation
\ref{eqn:w12pt} lies on $W_{12}$.
\end{cor}
\noindent We conclude this section by summarizing our method in the
Eigenform Location Algorithm (ELA, Figure \ref{fig:ELA}).  Our proof
of Theorem \ref{thm:w12pt} is to implement ELA with input defined over
$\qq$.  We can implement ELA with floating point input to numerically
sample the function in Equation \ref{eqn:a0I2} and generate the
polynomials in Table \ref{tab:wDrs} giving algebraic models of
Weierstrass curves.  In subsequent sections, we will implement ELA
with input defined over function fields over number fields to certify
our algebraic models of Weierstrass curves and prove Theorems
\ref{thm:wDrs} and \ref{thm:wDers}.

\paragraph{Igusa-Clebsch invariants.}  
For a genus two topological surface $S$, the Igusa-Clebsch invariants
define a holomorphic map $IC : \teich(S) \rightarrow \pp(2,4,6,10)$
where $\pp(2,4,6,10)$ is the weighted projective space
\begin{equation}
\label{eqn:wtprojspace}
\mathbb P(2,4,6,10) = \cc^4 / \cc^* \mbox{ where $\cc^*$ acts by }
\lambda \cdot (I_2,I_4,I_6,I_{10}) = (\lambda^2 I_2,\lambda^4
I_4,\lambda^6 I_6,\lambda^{10} I_{10}).
\end{equation}
The coordinate $I_k$ of $IC$ is called the {\em Igusa-Clebsch
  invariant of weight $k$} (cf. \cite{igusa:arithmeticmoduli}).

For $a \in V$ and $Y(a) \in \M_2$ as defined in Section
\ref{sec:algpairing}, the invariants $I_k(Y(a))$ are polynomial in
$a$.\footnote{We will not repeat the formula for $I_k$ here.  See
  \cite{igusa:arithmeticmoduli}, the function {\sf
    IgusaClebschInvariants()} in \texttt{Magma} or the file {\sf
    IIa.magma} in the auxiliary files.} The invariant $I_{10}$ is the
discriminant of the polynomial $f_a(w)$ and the curve $Y$ defined by
Equation \ref{eqn:w12pt} has
\[ IC(Y) = (56 : -32 : -348 : -324). \]
Just as the $j$-invariant gives an algebraic bijection between $\M_1$
and $\cc$, the Igusa-Clebsch invariants give an algebraic bijection
between $\M_2$ and $\cc^3$.  The map $IC: \teich(S) \rightarrow
\pp(2,4,6,10)$ is $\Mod(S)$ invariant and covers a bijection between
$\M_2$ and the hyperplane complement
\begin{equation}
\label{eqn:ICimage}
\left\{ (I_2 : I_4 : I_6 : I_{10}) : I_{10} \neq 0 \right\} \subset \pp(2,4,6,10).
\end{equation}

\paragraph{Real multiplication by $\ord_{12}$.}  
Recall from Section \ref{sec:jacRM} that the Hilbert modular surface
\[ X_{12} = \hh \times \hh / \PSL(\ord_{12} \oplus \ord_{12}^\vee) \]
admits an inverse period map $\jac_{12}^{-1} : X_{12} \to \M_2$ which
is a rational map parametrizing the collection of genus two surfaces
whose Jacobians admit real multiplication by $\ord_{12}$ (cf. Equation
\ref{eqn:invjac}).  Set $H_{12} = \cc^2$ (the $(r,s)$-plane) and
consider the map $IC_{12} : H_{12} \rightarrow \pp(2,4,6,10)$ defined
by 
\begin{multline}
\label{eqn:IC12}
IC_{12}(r,s) = \big( -8(3 + 3r - 3s - 2s^2 + 2s^3) : 4(s-1)^2(9r + 15rs + s^4) : \\
 -4(s-1)^2 (72r + 90r^2 + 48rs + 102r^2s - 141rs^2 - 38rs^3 + 8s^4 + 67rs^4 -8s^5 - 6s^6 + 6s^7)  \\
  : -4r^3(s-1)^6 \big).
\end{multline}
The map above is defined in \cite{elkieskumar:hms},
where it is shown that the map $X_{12} \to \pp(2,4,6,10)$ factors
through $IC_{12}$.
\begin{thm}[Elkies-Kumar]
\label{thmek:IC12}
The rational map $IC \circ \jac_{12}^{-1} : X_{12} \to \pp(2,4,6,10)$
is the composition of a degree two rational map $X_{12} \rightarrow
H_{12}$ branched along the curve
\begin{equation}
\label{eqn:b12}
b_{12}(r,s) = 0 \mbox{ where } b_{12}(r,s) = (s-1)(s+1)(16r + 27r^2 - 18rs^2 - s^4 + s^6)
\end{equation}
and the map $IC_{12} : H_{12} \to \pp(2,4,6,10)$.
\end{thm}
\noindent Compare Equation \ref{eqn:b12} with Table \ref{tab:bDrs}.
From Theorem \ref{thmek:IC12} we see that $\jac(Y)$ admits real
multiplication by $\ord_{12}$ if and only if $IC(Y)$ is in the closure
of the image of $IC_{12}$.
\begin{prop}
\label{prop:exhasRM}
The Jacobian of the genus two curve defined by Equation
\ref{eqn:w12pt} admits real multiplication by $\ord_{12}$.
\end{prop}
\begin{proof}  
The curve $Y$ defined by Equation \ref{eqn:w12pt} satisfies $IC(Y) =
IC_{12}(b)$ where $b=(-3/8,-1/2)$.
\end{proof}

\paragraph{Deformations.}  
Now set $a = (24,52,-8,-12,-2) \in V$ so that the algebraic curve
defined by Equation \ref{eqn:w12pt} is isomorphic to $Y(a)$ in the
notation of Section \ref{sec:algpairing} and set $b=(-3/8,-1/2)$ so that $IC(Y(a)) = IC_{12}(b)$.  Also, set
\begin{equation}
\label{eqn:avrvs}
 v_r = (80, -32, 112, -16, -12) \in T_aV \mbox{ and } v_s = (36, -6, 76, 13, -9) \in T_aV.
\end{equation}
These tangent vectors are solutions to the linear equations
\begin{equation}
\label{eqn:vrvs}
d(IC \circ Y)_a(v_r) = d(IC_{12})_b((1,0)) \text{ and } d(IC \circ Y)_a(v_s) = d(IC_{12})_b((0,1))
\end{equation} 
and were found by linear algebra.  There is a $3$-dimensional space of
solutions in each case; we chose integer solutions of small height.
\begin{prop} 
\label{prop:liftIC12}
There is an open neighborhood $B$ of $b = (-3/8,-1/2)$ in $\cc^2$ and
a holomorphic map $g : B \rightarrow V$ such that
\[ dg_{b}( (1,0) ) = v_r, dg_{b}( (0,1) ) = v_s \mbox{ and } IC\left( Y(g(r,s)) \right) = IC_{12}(r,s). \]
\end{prop}
\begin{proof}
This proposition follows the inverse function theorem and the equality
$IC(Y(a)) = IC_{12}(b)$, Equation \ref{eqn:vrvs} and the fact that $IC
\circ Y : V \to \pp(2,4,6,10)$ is a submersion at $a$.
\end{proof}
Using Theorem \ref{thm:algpairingcoeffs}, we can compute the
annihilator of the image of $dg_b$.
\begin{prop}
\label{prop:annw12pt}
The annihilator of the image of $dg_b$ is the line of quadratic
differentials spanned by $w \cdot dw^2/f_a(w)$ in $Q(Y(a))$.
\end{prop}
\begin{proof}
Setting $a = (24,52,-8,-12,-2)$, the matrix $M(a)$ defined in Theorem
\ref{thm:algpairingcoeffs} is
\begin{equation}
\label{eqn:w12ptMa}
M = \frac{1}{2^8 \cdot 3^6} \left(\begin{array}{rrrrr}
95 & -8 & 74 & 328 & 44 \\
-8 & 74 & 328 & 44 & 2752 \\
74 & 328 & 44 & 2752 & 5000
\end{array}\right).
\end{equation}
Let $L$ be the matrix with columns $v_r$ and $v_s$.  The nullspace of
$(M\cdot L)^T$ is spanned by $(0,1,0)$.  By Theorem
\ref{thm:algpairingcoeffs}, the annihilator of $dg_b$ is the line
spanned by
\[ (0 \cdot w^0 + 1 \cdot w^1+ 0 \cdot w^2) \frac{dw^2}{f_a(w)} = w \frac{dw^2}{f_a(w)}. \]
\end{proof}
\noindent We are now ready to prove Theorem \ref{thm:w12pt}.
\begin{proof}[Proof of Theorem \ref{thm:w12pt}]
Possibly making the neighborhood $B$ of $b$ in Proposition
\ref{prop:liftIC12} smaller, we can ensure that the map $g$
constructed in Proposition \ref{prop:liftIC12} has a lift $\widetilde
g: B \rightarrow \teich(S)$ where $S$ is a surface of genus two.  By
Theorem \ref{thmek:IC12}, we can choose a proper, self-adjoint
embedding $\rho : \ord \to \End(H_1(S,\zz))$ so that the image of
$\widetilde{g}$ is contained in $\teich(S,\rho)$.

Let $Y = \widetilde g(b)$ and, as in Section \ref{sec:jacRM}, let $U$
be the symplectic basis for $H_1(S,\rr)$ adapted to $\rho$ in the
sense of Proposition \ref{prop:eigenbasis} and let $\omega_1(Y)$ and
$\omega_2(Y)$ be the eigenforms dual to the $a$-cycles in $U$.  By
Proposition \ref{prop:eformproducts}, the product $\omega_1(Y) \cdot
\omega_2(Y)$ annihilates the image of $d\widetilde g_b$.  On the other
hand, $Y$ is biholomorphic to the algebraic curve defined by Equation
\ref{eqn:w12pt}, and in this model the annihilator of $d \widetilde
g_b$ is spanned by $w \cdot dw^2/f_a(w) = w \cdot dw/z \cdot dw/z$
(Proposition \ref{prop:annw12pt}).  Since there is, up to scale and
permutation, a unique pair of one-forms whose product is equal to $w
\cdot dw^2/f_a(w)$, the forms $dw/z$ and $w \cdot dw/z$ are eigenforms
for real multiplication by $\ord_{12}$ on $\jac(Y)$.
\end{proof}

\begin{figure}[h]
\setlength{\fboxsep}{0.25in}
\begin{center}
\fbox{
\parbox{5.3in}{
{\bf Eigenform Location Algorithm (ELA)}
\newline
\newline \noindent {\bf Input:} A triple $(IC_D,a,b)$ consisting of an
algebraic map $IC_D: H_D \to \pp(2,4,6,10)$, a point $a \in V$ and a
point $b \in H_D$.  The map $IC_D$ should satisfy an analogue of
Theorem \ref{thmek:IC12} for $X_D$ and the triple $(IC_D,a,b)$ should
satisfy the conditions in \ref{ELA:ICs} and \ref{ELA:ranks}.
\newline \noindent {\bf Output:} Eigenforms $[\omega_1,\omega_2]$ for $\ord_D$ on $Y(a)$.
\newline
\newline
\noindent If
\begin{enumerate}[leftmargin=*,label=\textbf{(ELA\arabic*)}]
\item \label{ELA:ICs} $IC_D(b) = IC(Y(a))$,
\item \label{ELA:ranks} $\operatorname{rank}(M(a))=\operatorname{rank}\left( d(IC \circ Y)_a \right) = 3$, and $\operatorname{rank}(d(IC_D)_b) = 2$.
\end{enumerate}
Then
\begin{enumerate}[resume,label=\textbf{(ELA\arabic*)},leftmargin=0.85in]
\item \label{ELA:computeL} compute a $5 \times 2$-matrix $L$ satisfying
\[ \operatorname{range}\left( d(IC \circ Y)_a \cdot L \right)= \operatorname{range}\left( d(IC_D)_b \right), \text{ and}\]
\item \label{ELA:computex} compute $x=(x_0,x_1,x_2)$ spanning the nullspace of $(M(a) \cdot L)^T$. 
\end{enumerate}
\hspace{0.225in} Return $\left[ \left( \lambda w +x_0 \right)\frac{dw}{z}, \left( x_2 w + \lambda \right) \frac{dw}{z} \right]$ for a non-zero $\lambda$ satisfying $\lambda^2+x_0x_2 = x_1 \lambda$.
\vspace{0.03in}
\newline Else print ``Error: (ELA1 and ELA2) is false'' and return None.
}
}
\end{center}
\sfcaption{\label{fig:ELA} The Eigenform Location Algorithm.}
\end{figure}
\paragraph{The Eigenform Location Algorithm.}  
We summarize our method of eigenform location in the Eigenform
Location Algorithm (ELA) in Figure \ref{fig:ELA}.  If \ref{ELA:ICs} is
true, we conclude that the Jacobian of $Y(a)$ admits real
multiplication by $\ord_D$.  The conditions on the ranks of $M(a)$ and
$d(IC \circ Y)_a$ in \ref{ELA:ranks} ensure that there is a
neighborhood $b$ in $H_D$ and a lift $g : B \to V$ with $g(b) = a$ as
in Proposition \ref{prop:liftIC12} and that we can compute the matrix
$L$ in \ref{ELA:computeL} by linear algebra.  The condition on the
rank of $d(IC_D)_b$ ensures the nullspace of $(M(a) \cdot L)^T$ is
one-dimensional.  Note that the product of the one-forms returned by
ELA is the differential $\lambda (x_2 w^2 + x_1 w + x_0) dw^2/f_a(w) =
\lambda q_x$.

\section{Weierstrass curve certification}
\label{sec:wcurvecertify}
In this section, we discuss implementing the Eigenform Location
Algorithm over function fields to give birational models for
irreducible Weierstrass curves.  We demonstrate this process in detail
for $W_{12}$, the first Weierstrass curve whose algebraic model has
not previously appeared in the literature.  We conclude with our proof
of Theorem \ref{thm:wDrs} giving birational models of irreducible
$W_D$.

We start with the following theorem which is a straightforward
application of ELA.
\begin{thm}
\label{thm:UCW12}
For generic $t \in \cc$, the Jacobian of the algebraic curve
$Y_{12}(t)$ defined by
\begin{multline}
\label{eqn:w12UC}
z^2 = w^5+ 2(2t + 3)w^4  -4(t - 1)(2t + 3)w^3 -8(t + 3)(2t + 3)^2 w^2  \\ 
+ 4(2t + 3)^2(t^2- 18t - 27) w +8(2t + 3)^3(t^2 + 14t + 21)
\end{multline}
admits real multiplication by $\ord_{12}$ with eigenforms $dw/z$ and
$w \cdot dw/z$.
\end{thm}

\begin{proof}
Define $a(t) = (a_0(t),\dots,a_4(t))$ so that $a_k(t)$ is the
coefficient of $w^k$ on the right hand side of Equation
\ref{eqn:w12UC} and $Y_{12}(t)$ is isomorphic to $Y(a(t))$.  Also set
$r(t) = -(13+10t+t^2)/t^3$, $s(t) = (t+3)/t$ and $b(t) = (r(t),s(t))$.
Running our Eigenform Location Algorithm with $IC_D = IC_{12}$, $b =
b(t)$ and $a = a(t)$ reveals that $dw/z$ and $w \cdot dw/z$ are
eigenforms for real multiplication by $\ord_{12}$ on $\jac(Y_{12}(t))$
for generic $t \in \cc$ (see {\sf cert12.magma} in the auxiliary
files).  Each of the steps in ELA is linear algebra in the field
$\qq(t)$.
\end{proof}
Since the form $dw/z \in \Omega(Y_{12}(t))$ has a double zero, the
one-form up to scale $\left( Y_{12}(t),[dw/z] \right)$ is in $W_{12}$
for most $t$.
\begin{cor}
The map $t \mapsto \left( Y_{12}(t),[dw/z] \right)$ defines a
birational map $h_{12} : \pp^1 \to W_{12}$.
\end{cor}
\begin{proof}
By Theorem \ref{thm:UCW12}, the pair $(Y_{12}(t),[dw/z])$ is in
$W_{12}$ for generic $t \in \cc$ and $t \mapsto \left(
Y_{12}(t),[dw/z] \right)$ defines a rational map $h_{12}: \pp^1 \to
W_{12}$.  To check that $h_{12}$ is birational, we compute the
composition $\pp^1 \xrightarrow{h_{12}} W_{12} \to \M_2
\xrightarrow{IC} \pp(2,4,6,10)$ and check (for instance, by computing
appropriate resultants) that it is non-constant and birational onto
its image. Since $W_{12}$ is irreducible \cite{mcmullen:spin}, we
conclude that $h_{12}$ birational.
\end{proof}
As a corollary, we can verify that the polynomial $w_{12}(r,s)$ in
Table \ref{tab:wDrs} gives a birational model for $W_{12}$ by checking
that $b(t)$ defines a birational map from $\pp^1$ to the curve
$w_{12}(r,s) = 0$, yielding the following proposition.
\begin{prop}
The immersion $W_{12} \to \M_2 \xrightarrow{IC} \pp(2,4,6,10)$ factors
through the composition of a birational map
\begin{equation}
W_{12} \to \left\{ (r,s) \in H_{12} : w_{12}(r,s)=0 \right\} \mbox{ where } w_{12}(r,s) = 27r + (8 - 12s - 9s^2 + 13s^3)
\end{equation}
and the map $IC_{12} : \cc^2 \to \pp(2,4,6,10)$ of Equation \ref{eqn:IC12}.
\end{prop}
We can now complete the proof of Theorem \ref{thm:wDrs} for most $D$.
\begin{proof}[Proof of Theorem \ref{thm:wDrs} (for $D \not \equiv 1 \bmod 8$)]
For each fundamental discriminant $D$ with $1 < D < 100$ and $D \not
\equiv 1 \bmod 8$, we provide two auxiliary computer files: {\sf
  ICDrs.magma} and {\sf certD.magma}.  In {\sf ICDrs.magma} we recall
the parametrization
\[ IC_D : H_D \to \pp(2,4,6,10) \]
defined in \cite{elkieskumar:hms} and satisfying an analogue of
Theorem \ref{thmek:IC12} for $X_D$.\footnote{For several
  discriminants, we change the coordinates given in
  \cite{elkieskumar:hms} by a product of M\"obius transformations on
  $H_D = \cc^2$ to simplify the equation for $W_D$.}  In all of our
examples, $H_D = \cc^2$.  In {\sf certD.magma}, we provide equations
for an algebraic curve $G_D$ over $\qq$ and define rational functions
\[ a_D : G_D \to V \mbox{ and } b_D : G_D \to H_D \]
where $b_D$ is birational onto the curve $w_D(r,s) = 0$ and $IC \circ
Y \circ a_D$ is birational onto its image.  We then call {\sf
  ELA.magma} which carries out ELA with $a = a_D$ and $b = b_D$,
certifying that
\begin{equation}
\label{eqn:GDtoWD}
 h_D(c) = (Y(a_D(c)),[dw/z]) \mbox{ defines a rational map } h_D : G_D \to W_D.
\end{equation}
Each of the steps in ELA is linear algebra in the field of algebraic
functions on $G_D$.  We conclude that the curves $w_D(r,s) =0$, $G_D$
and $W_D$ are birational to one another.
\end{proof}
 
\begin{rmk} 
An important ingredient in our proof of Theorem \ref{thm:wDrs} is an
explicit model of the universal curve over an open subset of $G_D$
(i.e. the function $a_D : G_D \to V$), which is not easy to compute
from $IC_D$ and $w_D(r,s)$.  The numerical sampling technique
described in Section \ref{sec:introduction} that we used to compute
$w_D(r,s)$ can also be used to sample the universal curve over $G_D$
and was used to generate the equations in {\sf certD.magma}.
\end{rmk}

Our proof of Theorem \ref{thm:wDrs} also proves Theorem \ref{thm:gDxy}
giving Weierstrass and plane quartic models for $W_D$ with $D \in
\left\{44,53,56,60,61 \right\}$.
\begin{proof}[Proof of Theorem \ref{thm:gDxy}]
For $D \in \left\{ 44,53,56,60,61 \right\}$, the curve $G_D$ defined
in {\sf certD.magma} is the curve defined by the equation $g_D(x,y) =
0$ (cf. Table \ref{tab:wpqmodels}) and, by the proof of Theorem
\ref{thm:wDrs}, is birational to $W_D$.
\end{proof}
\noindent Our proof of Theorem \ref{thm:wDrs} also proves the second
half of Theorem \ref{thm:rationalwd} concerning irreducible
Weierstrass curves of genus zero.
\begin{prop}
\label{prop:rationalwdoverq}
For $D \leq 41$ with $D \not \equiv 1 \bmod 8$ and $D \neq 21$, the
curve $W_D$ is birational to $\pp^1$ over $\qq$.
\end{prop}
\begin{proof}
For these discriminants, $G_D = \pp^1$ and the maps $a_D$, $b_D$
defined in {\sf certD.magma} are defined over $\qq$.
\end{proof}
\begin{prop}
\label{prop:w21conic}
The curve $W_{21}$ has no $\qq$-rational points and is birational over
$\qq$ to the conic $g_{21}(x,y) = 0$ where:
\begin{equation}
\label{eqn:g21}
g_{21}(x,y) = 21 \left( 11 x^2 -182 x-229 \right) +y^2. 
\end{equation}
\end{prop}
\begin{proof}
The curve $G_D$ defined in {\sf cert21.magma} is the conic defined by
$g_{21}(x,y) = 0$ and the maps $a_D$ and $b_D$ defined {\sf
  cert21.magma} are defined over $\qq$.  From the proof of Theorem
\ref{thm:wDrs}, we see that $W_{21}$ is birational over $\qq$ to the
curve $g_{21}(x,y) = 0$.  The closure of the conic $g_{21}(x,y) = 0$
in $\pp^2$ has no integer points, as can be seen by homogenizing
Equation \ref{eqn:g21} and reducing modulo $3$.  We conclude that
$W_{21}$ has no $\qq$-rational points.
\end{proof}

\section{Cusps and spin components}
\label{sec:cuspspin}
Using the technique described in Section \ref{sec:wcurvecertify} for
verifying our equations for irreducible $W_D$, we can also show that
the curve $w_D^0(r,s) = 0$ parametrizes {\em an} irreducible component
of reducible $W_D$.  We now turn to distinguishing the components of
$W_D$ by spin.

\paragraph{Cusps on Weierstrass curves.}  
Let $\overline{\M}_2$ be the Deligne-Mumford compactification of
$\M_2$ by stable curves and let $\overline{W}_D$ be the smooth
projective curve birational to $W_D$.  The curve $\overline{W}_D$ is
obtained from $W_D$ by smoothing orbifold points and filling in
finitely many cusps.  Since $\overline{W}_D$ and $\overline{\M}_2$ are
projective varieties, the map $W_D \to \M_2$ extends to an algebraic
map from $\overline{W}_D$ to the coarse space associated to
$\overline{\M}_2$.  The cusps of $\overline{W}_D$ are sent into
$\partial \overline{\M}_2$ under this map.

\paragraph{Locating cusps in birational models.}
The composition of $W_D \to \M_2 \xrightarrow{IC} \pp(2,4,6,10)$ also
extends to a map $\overline{W}_D \to \pp(2,4,6,10)$ and this extension
sends the cusps into the hyperplane $I_{10}=0$.  Given an explicit
algebraic curve $G_D$ and a rational map $a_D : G_D \to V$ giving rise
to the birational map $h_D : G_D \to W_D$ (cf. the proof of Theorem
\ref{thm:wDrs}), we can locate the smooth points in $G_D$
corresponding to cusps of $W_D$ by determining the poles of the
algebraic function
\begin{equation}
\label{eqn:I2I10}
c \mapsto (I_2 (Y(a_D(c)))^5 / I_{10}(Y(a_D(c))).
\end{equation}
 
\paragraph{Splitting prototypes.}  
The cusps on $W_D$ are enumerated in \cite{mcmullen:spin}.  A {\em
  splitting prototype of discriminant $D$} is a quadruple $(a,b,c,e)
\in \zz^4$ satisfying
\begin{equation}
\begin{array}{lll}
D = e^2+4bc, & 0 \leq a < \gcd(b,c), & c+e<b, \\
0 < b, &  0 < c,\mbox{ and} & \gcd(a,b,c,e) = 1.
\end{array}
\end{equation}
For example, the quadruple $(a,b,c,e) = (0,1,3,0)$ is a splitting
prototype of discriminant 12.
\begin{thm}[McMullen]
If $D$ is not a square, then the cusps of $W_D$ are in bijection with
the set of splitting prototypes of discriminant $D$.
\end{thm}

\paragraph{Stable limits and Igusa-Clebsch invariants.} 
Algebraic models of the singular curves corresponding to cusps of
$W_D$ are described in \cite{bainbridge:eulerchar} (see also
\cite{bouwmoeller:nonarithmetic}, Proposition 3.2).  From these models
it is easy to prove the following.
\begin{prop}
\label{prop:ICprot}
Let $(Y_n,[\omega_n]) \in W_D$ be a sequence tending to the cusp with
splitting prototype $p = (a,b,c,e)$.  Then $\lim_{n \to \infty}
IC(Y_n) = IC(p)$ where
\begin{multline}
\label{eqn:ICprot}
IC(p) = (12b^4 - 8b^3c + 12b^2c^2 - 4b^2e^2 + 24bce^2 + 6e^4 +e (3 e^2+3 D-4b^2)\sqrt{D} : b^4(e + \sqrt{D})^4 : \\
b^4(e + \sqrt{D})^4(4b^4 - 4b^3c + 4b^2c^2 - 2b^2e^2 + 8bce^2 + 2e^4 +e (e^2+ D- 2b^2)\sqrt{D} ) : 0).
\end{multline}
\end{prop}
\noindent For instance, with $Y_{12}(t)$ the algebraic curve defined
by Equation \ref{eqn:w12UC} we have
\[ \lim_{t \to \infty} IC(Y_{12}(t)) = (96 : 289 : 8092 : 0) = IC\left( (0,1,3,0) \right). \]

\paragraph{Spin invariant.}  
Now suppose $D \equiv 1 \bmod 8$.  For such discriminants, the curve
$W_D$ has two irreducible components $W_D^\epsilon$ distinguished by a
spin invariant $\epsilon \in \zz/2\zz$.
\begin{thm}[McMullen]
For a prototype $p = (a,b,c,e)$ of discriminant $D$ with $D \equiv 1
\bmod 8$, the cusp corresponding to $p$ lies on the spin
$\epsilon(p)$-component of $W_D$ where
\begin{equation}
\label{eqn:protspin}
\epsilon(p) = \frac{e-f}{2} + (c+1)(a+b+ab) \bmod 2
\end{equation}
and $f$ is the conductor of $\ord_D$.\footnote{The conductor of
  $\ord_D$ is the index of $\ord_D$ in the maximal order of $\ord_D
  \otimes_\zz \qq$.  Rings with fundamental discriminants such as
  those considered in this paper have conductor $f=1$.}
\end{thm}
We are now ready to prove Theorem \ref{thm:wDers} and complete the proof of Theorem \ref{thm:wDrs}.
\begin{proof}[Proof of Theorem \ref{thm:wDers}]
For each fundamental discriminant $1 < D < 100$ with $D \equiv 1 \bmod
8$, we provide computer files {\sf ICDrs.magma} and {\sf certD.magma}.
In {\sf ICDrs.magma} we recall the map $IC_D$ in
\cite{elkieskumar:hms} and in {\sf certD.magma} we define an algebraic
curve $G_D^0$ and rational functions
\[ a_D^0: G_D^0 \to V \mbox{ and } b_D^0 : G_D^0 \to H_D \]
so that $IC \circ Y \circ a_D^0$ is birational onto its image and
$b_D^0$ is birational onto the curve $w_D^0(r,s)=0$.  As in the proof
of Theorem \ref{thm:wDrs}, we then call {\sf ELA.magma} which
implements the Eigenform Location Algorithm verifying that
\begin{equation}
\label{eqn:hD0}
 h_D^0(c) = (Y(a_D^0(c)),[dw/z]) \mbox{ defines a rational map } h_D^0 : G_D^0 \to W_D.
\end{equation}
We conclude that $w_D^0(r,s) =0$ is birational to an irreducible
component of $W_D$.

We then identify a smooth point $c \in G_D^0$ and call {\sf
  spin\_check.magma} which checks that $c$ corresponds to a cusp of
$W_D$ (i.e. $I_2^5/I_{10}$ has a pole at $c$), identifies the
splitting prototypes $p$ of discriminant $D$ satisfying $IC(p) =
IC_D(b_D(c))$ and verifies that they all have even spin using Equation
\ref{eqn:protspin}.  This shows that the curve $w_D^0(r,s)=0$ is
birational to $W_D^0$.

Applying the non-trivial field automorphism of $\qq(\sqrt{D})$ to all
of the equations in {\sf certD.magma} gives a curve $G_D^1$ and maps
$a_D^1$, $b_D^1$ and $h_D^1$.  Since the equations in {\sf
  ICDrs.magma} have coefficients in $\qq$, ELA with input $a = a_D^1$
and $b=b_D^1$ will return the Galois conjugates of the eigenforms
returned by ELA with input $a=a_D^0$ and $b=b_D^0$.  We conclude that
the Galois conjugate $w_D^1$ of $w_D^0$ defines a curve birational to
another component of $W_D$.  We verify that this component is $W_D^1$
by the method above applied to the point in $G_D^1$ Galois conjugate
to $c \in G_D^0$.
\end{proof}

\begin{proof}[Proof of Theorem \ref{thm:wDrs} (for $D \equiv 1 \bmod 8$).] 
For $D \equiv 1 \bmod 8$, the fact that $w_D(r,s)=0$ gives a birational model for $W_D$ follows from Theorem \ref{thm:wDers} and the identity $w_D(r,s) = w_D^0(r,s) w_D^1(r,s)$, which most computer algebra systems will readily verify.
\end{proof}
\begin{rmk}
Some care has to be taken when choosing the point $c \in G_D^0$ in the
proof of Theorem \ref{thm:wDers} since the stable limit $h_D^0(c)$
does not always uniquely identify the corresponding splitting
prototype.  For instance, the first coordinate $a$ in the splitting
prototype does not affect the stable limit, as reflected by the fact
that $a$ does not appear on the right hand side of Equation
\ref{eqn:ICprot}.
\end{rmk}
We can combine the parametrization $h_D^0$ and its Galois conjugate
$h_D^1$ used in the proof of Theorem \ref{thm:wDers} into a birational
map
\begin{equation}
\label{eqn:hDred}
h_D : G_D = G_D^0 \sqcup G_D^1 \to W_D = W_D^0 \sqcup W_D^1.
\end{equation}
We will use $h_D$ in the next section to give biregular models of
reducible $\overline{W}_D$ for certain $D$ in the next section.

Our proof of Theorem \ref{thm:wDers} also establishes Theorem
\ref{thm:gDexy} which gives Weierstrass models for the components of
$W_{57}$, $W_{65}$ and $W_{73}$.
\begin{proof}[Proof of Theorem \ref{thm:gDexy}]
For $D \in \left\{ 57,65,73 \right\}$, the curve $G_D^0$ defined in
{\sf certD.magma} is the curve defined by $g_D^0(x,y) = 0$.  In the
proof of Theorem \ref{thm:wDers}, we saw that $W_D^0$ is birational to
$G_D^0$ and $W_D^1$ is birational to the Galois conjugate of $G_D^0$.
\end{proof}
We can also complete the proof of Theorem \ref{thm:rationalwd}
concerning Weierstrass curves of genus zero.
\begin{proof}[Proof of Theorem \ref{thm:rationalwd}]
For $D \leq 41$ with $D \equiv 1 \bmod 8$, the curve $G_D^0$ defined
in {\sf certD.magma} is $\pp^1$ and the maps $a_D$ and $b_D$ are
defined over $\qq(\sqrt{D})$.  This shows that the components of $W_D$
are birational to $\pp^1$ over $\qq(\sqrt{D})$ for such discriminants.
The remaining claims made in Theorem \ref{thm:rationalwd} are
established in Propositions \ref{prop:rationalwdoverq} and
\ref{prop:w21conic}.
\end{proof}

\section{Arithmetic geometry of Weierstrass curves}
\label{sec:arithmetic}
In this section, we study the arithmetic geometry of our examples of
Weierstrass curves and prove the remaining Theorems stated in Section
\ref{sec:introduction}.

\paragraph{Biregular models for Weierstrass curves.}  
For each fundamental discriminant $1 < D < 100$, we have now given a
birational parametrization $h_D$ of $W_D$ by an explicit algbraic
curve $G_D$ (cf. proofs of Theorems \ref{thm:wDrs} and
\ref{thm:wDers}).  The curve $G_D$ and parametrization $h_D : G_D \to
W_D$ are defined the auxiliary computer files.

Many of our birational models for small genus $W_D$ easily extend to
biregular models for the smooth, projective curve $\overline{W}_D$
birational to $W_D$.  For $D \leq 41$ with $D \neq 21$, $G_D$ is a
union of $k=1$ or $2$ projective $t$-lines and is already smooth and
projective, and $h_D$ extends to a biregular map $h_D :
\overline{G}_D=G_D \to \overline{W}_D$.  For $D \in \left\{
21,44,56,57,60,65,73 \right\}$, each irreducible component of $G_D$ is
an affine plane curve with smooth closure in $\pp^2$.  The birational
map $h_D$ extends to a biregular map $h_D : \overline{G}_D \to
\overline{W}_D$ where the irreducible components of $\overline{G}_D$
are disjoint and equal to the closures of irreducible components of
$G_D$ in $\pp^2$.  For $D \in \left\{ 53,61 \right\}$ the curve $G_D$
is an irreducible affine curve of genus two and has singular closure
in $\pp^2$.  The closure $\overline{G}_D$ of the algebraic set
\begin{equation}
\label{eqn:genustwoprojmodel}
 \left\{ \left( (x : y : 1) , (1/x : y/x^3 : 1) \right) : g_D(x,y) = 0, x \neq 0 \right\} \subset \pp^2 \times \pp^2
\end{equation}
is smooth, projective and birational to $G_D$ in an obvious way, and
the birational map $h_D$ naturally extends to a biregular map $h_D :
\overline{G}_D \to \overline{W}_D$.

For the remainder of this section, we will identify $\overline{W}_D$
for these discriminants ($D \leq 73$ with $D \neq 69$) with the
biregular models described above via the biregular map $h_D$.

\paragraph{Singular primes and primes of bad reduction.} 
Now that we have given smooth, projective models over $\zz$ for
several Weierstrass curves, we can study their primes of singular and
bad reduction.  For general discussion of these notions we refer the
reader to \cite{liu:alggeometry} (in particular \S 10.1.2) and
\cite{dalawat:badreduction}.  For an affine plane curve $C$ defined by
$g \in \zz[x,y]$ and a prime $p \in \zz$, we say that $p$ is a {\em
  prime of singular reduction for $C$} if the polynomial equations
\[ g = 0, \partial g/\partial x =0 \mbox{ and } \partial g / \partial y =0\]
have a simultaneous solution in an algebraically closed field of
characterstic $p$.  For a projective curve $C$ defined over $\zz$ and
covered by plane curves $C_1,\dots,C_n$ defined by polynomials
$g_1,\dots,g_n \in \zz[x,y]$, we will say that $p$ is a {\em prime of
  singular reduction for $C$} if $p$ is a prime of singular reduction
for at least one of the curves $C_k$.
\begin{table}
\begin{tabular}{cc}
\toprule \als
$D$ & Singular primes for $\overline{W}_D$ \\ \als
\midrule \als
$21$ & $\left\{ 2,3,5,7 \right\}$ \\ \als
$44$ & $\left\{2, 5, 11\right\}$ \\ \als
$53$ & $\left\{ 2,11,13,53 \right\}$ \\ \als
$56$ & $\left\{ 2,5,7,13 \right\}$ \\ \als
$60$ & $\left\{ 2, 3, 5, 7, 11 \right\}$ \\ \als
$61$ & $\left\{ 2,3,5,13,61 \right\}$ \\\als
\bottomrule
\end{tabular}
\sfcaption{ \label{tab:singularprimes} For $D \in \left\{
  44,53,56,60,61 \right\}$, the birational model $g_D(x,y) = 0$ for
  the Weierstrass curve $\overline{W}_D$ has a singularity at the
  prime $p$ for the primes listed above.}
\end{table}
For an affine or projective curve $C$ defined over $\qq$ and a prime
$p \in \zz$, we will call $p$ a prime of {\em bad reduction} for $C$
if $p$ is a singular prime for every curve $C'$ defined over $\zz$ and
biregular to $C$ over $\qq$.  In particular, the primes of singular
reduction for {\em any} integral model of $C$ contain the primes of
bad reduction of $C$.

\paragraph{Singular primes of low, positive genus Weierstrass curves.}  
As we demonstrate in our next proposition, the primes of singular
reduction for conic and hyperelliptic Weierstrass curves can be
computed using discriminants and the primes of singular reduction for
our genus three Weierstrass curves can be computed using elimination
ideals.
\begin{thm}
\label{thm:singularprimelist}
For $D \in \left\{ 21,44,53,56,60,61 \right\}$, the primes of singular
reduction for $\overline{W}_D$ are those listed in Table
\ref{tab:singularprimes}.
\end{thm}
\begin{proof}
For a hyperelliptic curve or conic birational to the plane curve
defined by a polynomial of the form $y^2+h(x)y+f(x) \in \zz[x,y]$, it
is standard to show that the primes of singular reduction are
precisely the primes dividing the discriminant of $h(x)^2-4f(x)$.
From this we easily verify that the primes listed in Table
\ref{tab:singularprimes} are the primes of singular reduction for
$\overline{W}_{21}$, $\overline{W}_{44}$, $\overline{W}_{53}$ and
$\overline{W}_{61}$.

Now set $D = 56$ or $60$ so that $\overline{W}_D$ is a smooth plane
quartic and let $g_D^h \in \zz[X,Y,Z]$ be the homogeneous, degree four
polynomial with $g_D^h(x,y,1) = g_D(x,y)$.  Also set $g_1(x,y) =
g_D^h(x,y,1)$, $g_2(x,y) = g_D^h(x,1,y)$ and $g_3(x,y) = g_D^h(1,x,y)$
so that $\overline{W}_D=C_1 \cup C_2 \cup C_3$ with $C_k$ the plane
curve defined by $g_k$.  For each of the primes listed next to $D$ in
Table \ref{tab:singularprimes}, we are able to find a simultaneous
solution to $g_k=0$, $\partial g_k/\partial x=0$ and $\partial
g_k/\partial y=0$ with coordinates in the finite field with $p$
elements for some $k$.  We conclude that each of these primes is a
prime of singular reduction for $\overline{W}_D$.  To show that there
are no other primes of singular reduction for $\overline{W}_D$, we
consider the elimination ideals
\[ E_k = I_k \cap \zz \mbox{ where } I_k = \left( g_k, \partial g_k/\partial x, \partial g_k / \partial y \right). \]
Elimination ideals can be computed using Gr\"obner bases and it is
easy to compute $E_k$ in \texttt{Magma}.  Clearly, if $p$ is a prime of
singular reduction for the affine curve $C_k$, then the ideal
generated by $p$ divides $E_k$.  The primes listed in Table
\ref{tab:singularprimes} are precisely those dividing $E_1 \cdot E_2
\cdot E_3$ and contain all of the primes of singular reduction for
$\overline{W}_D$.
\end{proof}
\noindent Theorem \ref{thm:singularprimes} about the primes of bad
reduction for certain Weierstrass curves is a corollary of Theorem
\ref{thm:singularprimelist}.
\begin{proof}[Proof of Theorem \ref{thm:singularprimes}]
The set primes of bad reduction for $\overline{W}_D$ is contained in
the set of primes of singular reduction for our biregular model of
$\overline{W}_D$.  By inspecting Table \ref{tab:singularprimes}, we
see that each prime of singular reduction for our model of
$\overline{W}_D$ divides the quantity $N(D)$ defined in Equation
\ref{eqn:ND}.
\end{proof}

\paragraph{Singular primes for genus zero Weierstrass curves.}  
We now turn to the Weierstrass curves $\overline{W}_D$ biregular to
the projective $t$-line $\pp^1$ over $\qq$.  In Section
\ref{sec:introduction}, we defined the {\em cuspidal polynomial} for
these curves to be the monic polynomial $c_D(t)$ vanishing simply at
the cusps of $\overline{W}_D$ in the affine $t$-line and non-zero
elsewhere.

As we described in Section \ref{sec:cuspspin}, we can locate the cusps
and compute $c_D(t)$ in each of these examples by determining the
poles of the algebraic function $I_2^5/I_{10}$ on $\overline{W}_D$.
We list the polynomials $c_D(t)$ along with their discriminants in
Table \ref{tab:cusppolys}, allowing us to prove Theorem
\ref{thm:cusppoly}.
\begin{table}
\begin{tabular}{ccc}
\toprule \als
$D$ & Cuspidal polynomial $c_D(t)$ & Discriminant of $c_D(t)$ \\ \als
\midrule \als
$5$ & $t-4$ & 1 \\ \als
$8$ & $t(t+1)$ & 1 \\ \als
$12$ & $t^2+10t+13$ & $2^4 \cdot 3$ \\ \als
$13$ & $t (t^2-14t-3)$ & $2^4 \cdot 3^2 \cdot 13$ \\ \als
$24$ & $t(t-16)(t^2-6)(t^2-24t-72)$ & $2^{36} \cdot 3^{14} \cdot 5^{14}$ \\ \als
$28$ & $(t^2 - 24 t - 423)(t^2 - 63)(t^2 + 14t + 21)$ & $2^{30} \cdot 3^{38} \cdot 7^7$ \\ \als
$29$ & $t (t^2-174t+145)(t^2+145 t-3625)$ & $2^{10} \cdot 5^{18} \cdot 7^8 \cdot 29^{10}$ \\ \als
$37$ & $\begin{array}{l} (t^2 - 2368)(t^2 - 1332)(t^2 + 74t + 1221) \hspace{0.5in}\\ 
\multicolumn{1}{r}{(t^3 + 51t^2 - 2220t - 114108)}\end{array}$ & $2^{60}\cdot 3^{23} \cdot 7^{32} \cdot 37^{28}$ \\ \als
$40$ & $\begin{array}{l} t(t + 81)(t^2 + 110t + 2025) \hspace{1in} \\ \multicolumn{1}{c}{(t^2 + 270t - 10935)(t^2 + 630t + 18225)} \\ \multicolumn{1}{r}{(t^3 + 351t^2 + 10935t + 164025)} \end{array}$ & $2^{168} \cdot 3^{267} \cdot 5^{66}$ \\ \als
\bottomrule
\end{tabular}
\sfcaption{\label{tab:cusppolys} For $\overline{W}_D$ birational to
  the projective $t$-line over $\qq$, the cuspidal polynomial $c_D(t)$
  is the polynomial vanishing simply at cusps of $\overline{W}_D$ in
  the finite $t$-line and nowhere else zero.}
\end{table}
\begin{proof}[Proof of Theorem \ref{thm:cusppoly}]
The polynomial $c_D(t)$ listed in Table \ref{tab:cusppolys} is
obviously in $\zz[t]$ and each of the primes dividing the discriminant
of $c_D(t)$ divides the quantity $N(D)$ defined in Equation
\ref{eqn:ND}.
\end{proof}

\paragraph{A Weierstrass elliptic curve.}  
The Weierstrass curve $\overline{W}_{44}$ is the only Weierstrass
curve associated to a fundamental discriminant and birational to an
elliptic curve over $\qq$.  From our explicit Weierstrass model for
$\overline{W}_{44}$, it is standard to compute various arithmetic
invariants and easy to do so in \texttt{Magma} or \texttt{Sage} (also
cf. \cite[\href{http://www.lmfdb.org/EllipticCurve/Q/880.i2}{Elliptic
    Curve 880.i2}]{lmfdb:880i2}).  We collect these facts about
$\overline{W}_{44}$ in the following proposition.
\begin{prop}
\label{prop:w44arithmeticinvariants}
The Weierstrass curve $\overline{W}_{44}$ has $j$-invariant
$j(\overline{W}_{44}) = 479^3/(11 \cdot 2^5 \cdot 5^5)$, conductor
$N\left( \overline{W}_{44} \right) = 880$, endomorphism ring
$\End\left(\overline{W}_{44}\right)$ isomorphic to $\zz$ and infinite
cyclic Mordell-Weil group $\overline{W}_{44}\left( \qq \right)$
generated by $(x,y)=(26,160)$.
\end{prop}
\begin{rmk}
We have numerical evidence, obtained using the functions related to
analytic Jacobians in \texttt{Magma}, that the endomorphism rings of
$\jac\left( \overline{W}_{53} \right)$ and $\jac\left(
\overline{W}_{61} \right)$ are also isomorphic to $\zz$.
\end{rmk}
Our identification of $\overline{W}_{44}$ with an elliptic curve turns
$\overline{W}_{44}$ into a group.  We will call the subgroup of
$\overline{W}_{44}$ generated by cusps the {\em cuspidal subgroup}.
By the method described in Section \ref{sec:cuspspin}, we can locate
the cusps on $\overline{W}_{44}$ and prove the following proposition.
\begin{prop}
\label{prop:cuspidalsubgroup44}
The cuspidal subgroup of $\overline{W}_{44}$ is freely generated by
\[  P_1 = \left(\frac{-38-48\sqrt{11}}{25},\frac{-1584+1936\sqrt{11}}{125}\right) \mbox{ and } P_2 = \left( 2+4\sqrt{11},44+16\sqrt{11} \right). \]
\end{prop}
\begin{proof}
The second column of Table \ref{tab:w44cusps} identifies the locations
of the cusps for $\overline{W}_{44}$ in our elliptic curve model
$g_{44}(x,y)=0$ and the fourth column asserts relations among these
points in the group law (e.g. the cusp at $Q = (-9,10\sqrt{11})$ is
equal to $6P_1-9P_2$).  It is standard to verify these relations and
easy to do so in \texttt{Magma}.  We conclude that the cuspidal subgroup
is generated by $P_1$ and $P_2$.

To show that the cuspidal subgroup is freely generated by $P_1$ and
$P_2$, we first check that $P_1-P_2$ is a $\qq$-rational point.  By
Proposition \ref{prop:w44arithmeticinvariants}, the difference
$P_1-P_2$ generates a free subgroup of $\overline{W}_{44}$.  Next, we
check that $n \cdot P_2$ is not $\qq$-rational for any $n \leq 18$.
By Kamienny's bound on the torsion order of points on elliptic curves
over quadratic fields \cite{kamienny:torsion}, we conclude that $P_1$
and $P_1-P_2$ generate a free subgroup of $\overline{W}_{44}$ and the
proposition follows.
\end{proof}
Theorem \ref{thm:w44mw} concerning the subgroup of $\pic^0\left(
\overline{W}_{44} \right)$ generated by pairwise cusp differences is
an immediate corollary.
\begin{proof}[Proof of Theorem \ref{thm:w44mw}]
Since the identity $(x,y) = (\infty,\infty)$ in $\overline{W}_{44}$ is
a cusp, the cuspidal subgroup is naturally isomorphic to the subgroup
of $\pic^0\left( \overline{W}_{44} \right)$ generated by pairwise cusp
differences.  By Proposition \ref{prop:cuspidalsubgroup44}, the
cuspidal group is isomorphic to $\zz^2$.
\end{proof}
\begin{center}
\begin{table}
\begin{tabular}{cccc}
\toprule
\als
Prototype & $(x,y)$ & $(r,s)$ & Mordell-Weil \\
\als
\midrule
$(0,11,1,0)$ & $(\infty,\infty)$ & $(-1,0)$ & $(0,0)$ \\ \als
$(0,7,1,4)$  & $\left(\frac{1}{25}(-38-48\sqrt{11}),\frac{1}{125}(-1584+1936\sqrt{11})\right)$ & $(-1,0)$ & $(1,0)$  \\ \als
$(0,7,1,-4)$ & $\left(\frac{1}{25}(-38+48\sqrt{11}),\frac{1}{125}(-1584-1936\sqrt{11}) \right)$ & $(-1,0)$ & $(5,-6)$ \\ \als
$(0,5,2,-2)$ & $\left( 2+4\sqrt{11},44+16\sqrt{11} \right)$ & $(1,0)$ & $(0,1)$ \\ \als
$(0,5,2,2)$ & $\left( 2-4\sqrt{11},44-16\sqrt{11} \right)$ & $(1,0)$ & $(4,-5)$ \\ \als
$(0,2,1,-6)$ & $(-9,10\sqrt{11})$ & $\left( \frac{1}{15}(2-2\sqrt{11}),0 \right)$ & $(6,-9)$ \\ \als
$(0,1,2,-6)$ & $(-9,-10\sqrt{11})$ & $\left( \frac{1}{15}(2+2\sqrt{11}),0 \right)$ & $(-6,9)$ \\ \als
$(0,10,1,-2)$ & $\left( 66+20\sqrt{11},740+240\sqrt{11} \right)$ & $(-1,0)$ & $(-2,4)$ \\ \als
$(0,10,1,2)$ & $\left( 66-20\sqrt{11},740-240\sqrt{11} \right)$ & $(-1,0)$ & $(6,-8)$ \\ \als
\bottomrule
\end{tabular}
\sfcaption{\label{tab:w44cusps} For each of the nine splitting prototypes of discriminant $44$, we list the $(x,y)$ coordinates in the Weierstrass model $g_{44}(x,y)=0$, the $(r,s)$-coordinates in the $w_{44}(r,s) = 0$ model and the Mordell-Weil coordinates in the cuspidal subgroup of $\overline{W}_{44}$ for the corresponding cusp.}
\end{table}
\end{center}
\unskip
\paragraph{Canonical divisors supported at cusps.}  
A genus two curve with Weierstrass model given by $y^2 + h(x) y +
f(x)=0$ admits a hyperelliptic involution $\eta$ given by the formula
$\eta(x,y)=(x,-h(x)-y)$.  The orbits of $\eta$ are intersections with
vertical lines $x=c$ and canonical divisors.  By locating the cusps on
$\overline{W}_{53}$ as described in Section \ref{sec:cuspspin}, we
find two canonical divisors supported at cusps.
\begin{prop}
\label{prop:w53canonicaldivisors}
The holomorphic one-forms on $\overline{W}_{53}$ given by
\[ \omega_1 = \left(2x+7-2\sqrt{53}\right) dx/y \mbox{ and }\omega_2 = \left(2x+7+2\sqrt{53}\right) dx/y\]
vanish only at cusps.
\end{prop}
\noindent By contrast, after computing the cusp locations on
$\overline{W}_{61}$, we find that there are no such forms on
$\overline{W}_{61}$.
\begin{prop}
\label{prop:w61canonicaldivisors}
There are no holomorphic one-forms on $\overline{W}_{61}$ which vanish
only at cusps.
\end{prop}
\noindent For both $\overline{W}_{53}$ and $\overline{W}_{61}$, the
hyperelliptic involution $\eta$ does not preserve the set of cusps,
yielding our next proposition.
\begin{prop}
\label{prop:genustwohypinvolution}
For $D \in \left\{ 53,61 \right\}$, the hyperelliptic involution on
$\overline{W}_D$ does not restrict a hyperbolic isometry of $W_D$.
\end{prop}
Our smooth plane quartic Weierstrass curves---$\overline{W}_{56}$ and
$\overline{W}_{60}$---are canonically embedded in $\pp^2$.  In
particular, intersections with lines are canonical divisors.  By
computing the cusp locations on $\overline{W}_{56}$, we find a
canonical divisor supported at cusps.  In the following propositions,
we let $X$, $Y$ and $Z$ be homogeneous coordinates on the projective
closure of the $(x,y)$-plane, with $x = X/Z$ and $y=Y/Z$
\begin{prop}
\label{prop:w56canonicaldivisors}
The line $Y = 2Z$ meets $\overline{W}_{56}$ at a canonical divisor
supported at cusps.
\end{prop}
\noindent On $\overline{W}_{60}$, we find five canonical divisors
supported at cusps.
\begin{prop}
\label{prop:w60canonicaldivisors}
Each of the following five lines 
\begin{equation}
\begin{array}{c}
Y=0, 4X +(6-\sqrt{60}) Y =0, 4X +(6+\sqrt{60})Y = 0,\\ -6X+10X-Z=0 \mbox{ and } 6X+5X+Z =0
\end{array}
\end{equation}
meets $\overline{W}_{60}$ at a canonical divisor supported at cusps.
\end{prop}
Combining the propositions of this paragraph, we can now prove Theorem
\ref{thm:canonicaldivisors}.
\begin{proof}[Proof of Theorem \ref{thm:canonicaldivisors}]
By Propositions \ref{prop:w53canonicaldivisors},
\ref{prop:w56canonicaldivisors} and \ref{prop:w60canonicaldivisors},
each of the curves $\overline{W}_{53}$, $\overline{W}_{56}$ and
$\overline{W}_{60}$ has a canonical divisor supported at cusps.  By
Proposition \ref{prop:w61canonicaldivisors}, the curve
$\overline{W}_{61}$ has no canonical divisor supported at cusps.
\end{proof}

\paragraph{Principal divisors supported at cusps.}  
We now prove the following theorem about principal divisors supported
at cusps on Weierstrass curves.
\begin{prop}
\label{prop:principaldivisors}
Each of the curves $\overline{W}_{44}$, $\overline{W}_{53}$,
$\overline{W}_{57}$, $\overline{W}_{60}$, $\overline{W}_{65}$ and
$\overline{W}_{73}$ has a principal divisor supported at cusps.
\end{prop}
\begin{proof}
By Propositions \ref{prop:w53canonicaldivisors} and
\ref{prop:w60canonicaldivisors}, each of the curves
$\overline{W}_{53}$ and $\overline{W}_{60}$ has a pair of holomorphic
one-forms which are distinct up to scale and vanish only at cusps.
The ratio of these two one-forms defines an algebraic function with
zeros and poles only at cusps.

The irreducible components of the remaining curves all have genus one
and our biregular models for these curves are elliptic curves.  By the
technique described in Section \ref{sec:cuspspin}, we locate their
cusps.  We find that the identity $(x,y) = (\infty,\infty)$ is a cusp
in each case and then search for (and find) relations among the cusps
in the group law by computing small integer combinations among triples
of cusps.  We have already given many such relations for
$\overline{W}_{44}$ in Table \ref{tab:w44cusps}.  For $D \in \left\{
57,65,73 \right\}$, we include a comment in {\sf certD.magma}
identifying the locations of several cusps and a relation among them.
\end{proof}
Theorem \ref{thm:principaldivisors} is an immediate corollary of
Proposition \ref{prop:principaldivisors}.
\begin{proof}[Proof of Theorem \ref{thm:principaldivisors}]
For a projective curve $C$ and a finite set $S \subset C$, the curve
$C$ has a principal divisor supported at $S$ if and only if $C
\setminus S$ admits a non-constant holomorphic map to $\cc^*$.  By
Proposition \ref{prop:principaldivisors}, for each $D \in \left\{
44,53,57,60,65,73 \right\}$, the curve $\overline{W}_D$ has a
principal divisor supported at its cusps $S_D \subset \overline{W}_D$,
so the curve $W_D = \overline{W}_D \setminus S_D$ admits a
non-constant holomorphic map to $\cc^*$.
\end{proof}

\appendix
\addtocounter{section}{19}
\section{Tables} \label{app:tables}
In this Appendix, we provide tables listing birational models of the
Hilbert modular surface $X_D$ (Table \ref{tab:bDrs}), the Weierstrass
curve $W_D$ (Tables \ref{tab:wDrs} and \ref{tab:wD0rs}) and the
homeomorphism type of $W_D$ (Table \ref{tab:homeotype}) for
fundamental discriminants $1 < D< 100$. For brevity, we only include
part of the first three tables; the full list of equations for the
Hilbert modular surfaces and the (components of) Weierstrass curves
are available in the computer files.

\begin{table}[h!]
{\centering
\begin{tabular}{c}
\toprule
\als
\large Algebraic models of Hilbert modular surfaces\\
\als
\midrule
\ppoly{b_5(r,s) = 972 r^5+324 r^4+27 r^3+4500 r^2 s+1350 r s-6250 s^2+108 s} \\ \als
\ppoly{b_8(r,s) = 16 r^3+32 r^2 s+24 r^2+16 r s^2-40 r s+12 r-s+2} \\ \als
\ppoly{b_{12}(r,s) = 27 r^2 s^2-27 r^2-18 r s^4+34 r s^2-16 r+s^8-2 s^6+s^4} \\ \als
\ppoly{b_{13}(r,s) = 128 r^3+27 r^2 s^2-656 r^2 s-192 r^2-108 r s^3+468 r s^2-568 r s+96 r-4 s^2+16 s-16} \\ \als
\ppoly{b_{17}(r,s) = 4 r^6+20 r^5-48 r^4 s+41 r^4+236 r^3 s+44 r^3+192 r^2 s^2+346 r^2 s+26 r^2+464 r s^2+144 r s+8 r-256 s^3+185 s^2+18 s+1} \\ \als
\ppoly{b_{21}(r,s) = 189 r^6-594 r^5 s+621 r^4 s^2-378 r^4-216 r^3 s^3+1116 r^3 s-954 r^2 s^2+205 r^2+184 r s^3-522 r s+16 s^4+349 s^2-16} \\ \als
\vdots \\ \als
\bottomrule
\als
\end{tabular}
\sfcaption{\label{tab:bDrs} {\sf The Hilbert modular surface $X_D$ is birational to the degree two cover of the $(r,s)$-plane branched along the curve $b_D(r,s)=0$.}}
}
\end{table}

\begin{table}[h!]
{\centering
\begin{tabular}{c}
\toprule
\als
\large Algebraic models of Weierstrass curves \\
\als
\midrule
\ppoly{w_{5}(r,s)=15 r+2}\\ \als
\ppoly{w_{8}(r,s)=4 r+4 s+1}\\ \als
\ppoly{w_{12}(r,s) = 27r + (8 - 12s - 9s^2 + 13s^3) } \\ \als
\ppoly{w_{13}(r,s)=26 r+108 s^2-252 s-9}\\ \als
\ppoly{w_{17}(r,s)=18 + 102r + 196r^2 + 136r^3 + 16r^4 + 288s + 544rs - 64r^2s - 1024s^2}\\ \als
\ppoly{w_{21}(r,s)=108 r^4-216 r^3 s-513 r^3+108 r^2 s^2+621 r^2 s-925 r^2-108 r s^2+1650 r s+205 r-225 s^2+795 s+500}\\ \als
\vdots \\ \als
\bottomrule
\als
\end{tabular}
\sfcaption{\label{tab:wDrs} For discriminants $1 < D < 100$, the Weierstrass curve is birational to the plane curve $w_D(r,s) = 0$.} 
}
\end{table}

\begin{table}[h!]
{\centering
\begin{tabular}{c}
\toprule
\als
{\large Algebraic models of reducible Weierstrass curves}  \\
\als
\midrule
\als
\ppoly{w_{17}^0(r,s) = (2+2 \sqrt{17})r^2 -(17-7 \sqrt{17})r +64 s-(9-3 \sqrt{17})} \\ \als
\ppoly{w_{33}^0(r,s) = 36r^3 -36r^2 s-(162+18 \sqrt{33})r^2 -36r s^2+(63+15 \sqrt{33})r s+(447+63 \sqrt{33})r + 36 s^3+(99+3 \sqrt{33}) s^2-(213+21 \sqrt{33}) s+(42+10 \sqrt{33}) } \\ \als
\ppoly{w_{41}^0(r,s) = 16r^3 s^2+8r^3 s+1r^3 +(864+160 \sqrt{41})r^2 s^2-(154-2 \sqrt{41})r^2 s-(-8 \sqrt{41})r^2 - (7680+1280 \sqrt{41})r s^2+(2288+272 \sqrt{41})r s+80r +(15872+2560 \sqrt{41}) s^2-(7200+1120 \sqrt{41}) s} \\ \als
\ppoly{w_{57}^0(r,s) = 576r^3 s^3+(864-96 \sqrt{57})r^3 s^2-(504+24 \sqrt{57})r^3 s-(792-72 \sqrt{57})r^3 -288r^2 s^3+(2304+288 \sqrt{57})r^2 s^2-(2892+348 \sqrt{57})r^2 s+(228+148 \sqrt{57})r^2 -144r s^3-(828+60 \sqrt{57})r s^2+(3294+486 \sqrt{57})r s- (3078+302 \sqrt{57})r +72 s^3-(270+30 \sqrt{57}) s^2+(1083+159 \sqrt{57}) s-(1083+95 \sqrt{57})} \\ \als
\vdots \\ \als
\bottomrule
\end{tabular}
\caption{\label{tab:wD0rs} {\sf For $1 < D < 100$ with $D \equiv 1 \bmod 8$, the curve $W_D^0$ is birational the curve $w_D^0(r,s)=0$ and $W_D^1$ is the Galois conjugate of $W_D^0$.}}
}
\end{table}

\begin{table}[h!]
\begin{equation*}
\begin{array}{ccccc@{\hskip 2em}ccccc}
\toprule
\als
D & g & e_2 & C & \chi & D & g & e_2 & C & \chi \\
\als
\midrule
 5^* & 0 & 1 & 1 & -\frac{3}{10} & 56 & 3 & 2 & 10 & -15 \\ \als
 8^* & 0 & 0 & 2 & -\frac{3}{4} & 57 & \{1,1\} & \{1,1\} & \{10,10\} & \left\{-\frac{21}{2},-\frac{21}{2}\right\} \\ \als
 12 & 0 & 1 & 3 & -\frac{3}{2} & 60 & 3 & 4 & 12 & -18 \\ \als
 13 & 0 & 1 & 3 & -\frac{3}{2} & 61 & 2 & 3 & 13 & -\frac{33}{2} \\ \als 
 17 & \{0,0\} & \{1,1\} & \{3,3\} & \left\{-\frac{3}{2},-\frac{3}{2}\right\} & 65 & \{1,1\} & \{2,2\} & \{11,11\} & \{-12,-12\} \\ \als
 21 & 0 & 2 & 4 & -3 & 69 & 4 & 4 & 10 & -18 \\ \als
 24 & 0 & 1 & 6 & -\frac{9}{2} & 73 & \{1,1\} & \{1,1\} & \{16,16\} & \left\{-\frac{33}{2},-\frac{33}{2}\right\} \\ \als
 28 & 0 & 2 & 7 & -6 & 76 & 4 & 3 & 21 & -\frac{57}{2} \\ \als
 29 & 0 & 3 & 5 & -\frac{9}{2}&  77 & 5 & 4 & 8 & -18 \\ \als
 33 & \{0,0\} & \{1,1\} & \{6,6\} & \left\{-\frac{9}{2},-\frac{9}{2}\right\} & 85 & 6 & 2 & 16 & -27 \\ \als
 37 & 0 & 1 & 9 & -\frac{15}{2} & 88 & 7 & 1 & 22 & -\frac{69}{2} \\ \als
 40 & 0 & 1 & 12 & -\frac{21}{2} & 89 & \{3,3\} & \{3,3\} & \{14,14\} & \left\{-\frac{39}{2},-\frac{39}{2}\right\} \\ \als
 41 & \{0,0\} & \{2,2\} & \{7,7\} & \{-6,-6\} & 92 & 8 & 6 & 13 & -30 \\ \als
 44 & 1 & 3 & 9 & -\frac{21}{2} & 93 & 8 & 2 & 12 & -27 \\ \als
 53 & 2 & 3 & 7 & -\frac{21}{2} & 97 & \{4,4\} & \{1,1\} & \{19,19\} & \left\{-\frac{51}{2},-\frac{51}{2}\right\} \\ \als
 \bottomrule
 \end{array}
\end{equation*}
\sfcaption{\label{tab:homeotype} For discriminants $D > 8$, the homeomorphism type of each irreducible component of $W_D$ is determined by its genus $g$, the number of cusps $C$ and the number of points of orbifold order two $e_2$.  For reducible $W_D$, the two irreducible components are homeomorphic and we list their topological invariants separately.   The curves $W_5$ and $W_8$ are isomorphic to the $(2,5,\infty)-$ and $(4,\infty,\infty)-$orbifolds respectively.}
\end{table}

\clearpage

\bibliographystyle{plain}

\begin{thebibliography}{10}
\bibitem{ahlfors:complexanalytic} 
Ahlfors, L. V.
``The complex analytic structure of the space of closed Riemann surfaces.''
In \textit{Analytic Functions}, pages 45--66. 
Princeton University Press, 1960.

\bibitem{bainbridge:eulerchar}
Bainbridge, M.
``Euler characteristics of Teichm\"uller curves in genus two.''
\textit{Geom. Topol.}, 11:1887--2073, 2007.

\bibitem{bainbridgemoller:deligne}
Bainbridge, M., and M\"oller, M.
``The Deligne--Mumford compactification of the real multiplication
locus and Teichm\"uller curves in genus 3.''
\textit{Acta mathematica}, 208(1):1--92, 2012.

\bibitem{birkenhakelange:cxabelianvarieties}
Birkenhake, C., and Lange, H.
\textit{Complex Abelian Varieties}.
Springer-Verlag, Berlin, second edition, 2004.

\bibitem{magma}
Bosma, W., Cannon, J., and Playoust, C..
``The Magma algebra system. I. The user language.''
\textit{J. Symbolic Comput.}, 24(3-4):235--265, 1997.
Computational algebra and number theory (London, 1993).

\bibitem{bouwmoeller:nonarithmetic}
Bouw, I., and M\"oller, M.,
``Differential equations associated with nonarithmetic Fuchsian
  groups.''
\textit{J. Lond. Math. Soc. (2)}, 81(1):65--90, 2010.
  
\bibitem{bouwmoeller:triangleveechgps}
Bouw, I., and M\"oller, M.,
``Teichm\"uller curves, triangle groups, and Lyapunov exponents.''
\textit{Ann. of Math. (2)}, 172(1):139--185, 2010.

\bibitem{calta:periodicity}
Calta, K.,
``Veech surfaces and complete periodicity in genus two.''
\textit{J. Amer. Math. Soc.}, 17(4):871--908, 2004.

\bibitem{dalawat:badreduction}
Dalawat, C. S.,
``Good reduction, bad reduction.''
{\em Preprint}, 2006, \arXiv{math/0605326}.

\bibitem{drinfeld:modularcurves}
Drinfeld, V. G.,
``Two theorems on modular curves.''
\textit{Funkcional. Anal. i Prilo\v{z}en},
7:83--84, 1973. 
English translation: \textit{Functional Anal. Appl.} 7:155--156, 1973.

\bibitem{elkieskumar:hms}
Elkies, N., and Kumar, A.,
``K3 surfaces and equations for {H}ilbert modular surfaces.''
\textit{Algebra Number Theory}, 8(10):2297--2411, 2014.

\bibitem{Gross-Zagier}
Gross, B., and Zagier, D.,
``On singular moduli.''
\textit{J. Reine Angew. Math}, 355(2):191--220, 1985.

\bibitem{gruenewald:humbert}
Gruenewald, D.,
\textit{Explicit algorithms for Humbert surfaces}.
PhD thesis, University of Sydney, 2008.

\bibitem{harris:moduli}
Harris, J., and Morrison, I.,
\textit{Moduli of Curves}.
Springer-Verlag, New York, NY, 1998.

\bibitem{hirzebruchzagier:intersectionnos}
Hirzebruch, F., and Zagier, D.,
``Intersection numbers of curves on hilbert modular surfaces and
modular forms of nebentypus.''
\textit{Inventiones mathematicae}, 36(1):57--113, 1976.
  
\bibitem{hubbard:teichthy}
Hubbard, J. H.,
\textit{Teichm\"uller Theory and Applications to Geometry,
  Topology, and Dynamics}, volume~1.
Matrix Editions, Ithaca, NY, 2006.

\bibitem{hubbardschleicher:spider}
Hubbard, J. H., and Schleicher, D.,
``The spider algorithm.''
In \textit{Complex Dynamical Systems}, volume~49 of {\em Proc. Sympos.
  Appl. Math.}, pages 155--180. Amer. Math. Soc., Providence, RI, 1994.

\bibitem{igusa:arithmeticmoduli}
Igusa, J.,
``Arithmetic variety of moduli for genus two.''
\textit{Ann. of Math. (2)}, 72:612--649, 1960.

\bibitem{imayoshitaniguchi:teichthy}
Imayoshi, Y., and Taniguchi, M.,
\textit{An Introduction to Teichm\"uller Spaces}.
Springer-Verlag, Tokyo, 1992.

\bibitem{kamienny:torsion}
Kamienny, S.,
``Torsion points on elliptic curves.''
\textit{Bull. Amer. Math. Soc.}, 23(2):371--373, 1990.

\bibitem{km:correspondences}
Kumar, A., and Mukamel, R. E.,
``Real multiplication through explicit correspondences.''
In \textit{ANTS XII: Proceedings of the Twelfth Algorithmic Number
  Theory Symposium}, 2016.

\bibitem{liu:alggeometry}
Liu, Q.,
\textit{Algebraic Geometry and Arithmetic Curves}.
Oxford University Press, Oxford, 2002.
Translated by R. Ern\'e.

\bibitem{lmfdb:880i2}
The {LMFDB Collaboration},
The L-functions and Modular Forms Database.
\url{http://www.lmfdb.org}.
Retrieved 2014.

\bibitem{lochak:arithmetic}
Lochak, P.,
``On arithmetic curves in the moduli space of curves.''
\textit{J. Inst. Math. Jussieu}, 4(3):443--508, 2005.

\bibitem{manin:parabolicpts}
Manin, J. I.,
``Parabolic points and zeta functions of modular curves.''
\textit{Izv. Akad. Nauk SSSR Ser. Mat.}, 36:19--66, 1972. 
English translation: \textit{USSR-Izv.} 6:19--64, 1972.

\bibitem{maple}
Maple 18,
Maplesoft, a division of Waterloo Maple Inc., 2014.

\bibitem{masurtabachnikov:billiards}
Masur, H., and Tabachnikov, S.,
``Rational billiards and flat structures.''
In \textit{Handbook of dynamical systems, {V}ol.\ 1{A}}, pages
1015--1089. North-Holland, Amsterdam, 2002.

\bibitem{matheuswright:htplanes}
Matheus, C., and Wright, A.,
``Hodge--Teichm\"uller planes and finiteness results for
Teichm\"uller curves.''
\textit{Duke Math. J.}, 164(6):1041--1077, 2015.

\bibitem{maxima}
Maxima.
Maxima, a computer algebra system. version 5.34.1, 2014.

\bibitem{mcmullen:billiards}
McMullen, C. T.,
``Billiards and Teichm\"uller curves on Hilbert modular surfaces.''
\textit{J. Amer. Math. Soc.}, 16(4):857--885, 2003.

\bibitem{mcmullen:spin}
McMullen, C. T.,
``Teichm\"uller curves in genus two: Discriminant and spin.''
\textit{Math. Ann.}, 333(1):87--130, 2005.

\bibitem{mcmullen:prym}
McMullen, C. T.,
``Prym varieties and Teichm\"uller curves.''
\textit{Duke Math. J.}, 133(3):569--590, 2006.

\bibitem{mcmullen:torsion}
McMullen, C. T.,
``Teichm\"uller curves in genus two: Torsion divisors and ratios of
sines.''
\textit{Invent. Math.}, 165(3):651--672, 2006.

\bibitem{mcmullen:foliations}
McMullen, C. T., 
``Foliations of {H}ilbert modular surfaces.''
\textit{Amer. J. Math.}, 129:183--215, 2007.

\bibitem{mcmullen:navigating}
McMullen, C. T.,
``Navigating moduli space with complex twists.''
\textit{J. Eur. Math. Soc.}, 15:1223--1243, 2013.

\bibitem{moller:torsion}
M{\"o}ller, M.,
``Periodic points on Veech surfaces and the Mordell-Weil group
over a Teichm{\"u}ller curve.''
\textit{Invent. Math.}, 165(3):633-649, 2006.
  
\bibitem{moller:variation}
M{\"o}ller, M.,
``Variations of Hodge structures of a Teichm\"uller curve.''
\textit{J. Amer. Math. Soc.}, 19(2):327--344, 2006.

\bibitem{mukamel:fundamentaldomains}
Mukamel, R. E.,
``Fundamental domains and generators for lattice Veech groups.''
\textit{Comment. Math. Helv.}, to appear.

\bibitem{mukamel:orbpts}
Mukamel, R. E.,
``Orbifold points on Teichm\"uller curves and Jacobians with
complex multiplication.''
\textit{Geom. Topol.}, 18(2):779--829, 2014.

\bibitem{pilgrim:coursenotes}
Pilgrim, K.,
``Riemann surfaces, dynamics, groups, and geometry (course notes).''
\url{http://mypage.iu.edu/~pilgrim/Teaching/M731F2008.pdf}, Retrieved
2014.

\bibitem{runge:endomorphism}
Runge, B.,
``Endomorphism rings of abelian surfaces and projective models of their
moduli spaces.''
\textit{Tohoku Math. J.}, 51:283--304, 1999.

\bibitem{Silverman-Elliptic}
Silverman, J. H.,
\textit{The Arithmetic of Elliptic Curves}, volume 106 of {\em
  Graduate Texts in Mathematics}.
Springer, second edition, 2009.

\bibitem{sage}
Stein, W. A., et~al.,
\textit{Sage Mathematics Software (Version 6.3)}.
The Sage Development Team, 2014.
\url{http://www.sagemath.org}.

\bibitem{pari}
The PARI~Group, Bordeaux.
\textit{PARI/GP version {\tt 2.5.5}}, 2014.
available from \url{http://pari.math.u-bordeaux.fr/}.

\bibitem{vdgeer:hms}
van~der Geer, G.,
\textit{Hilbert Modular Surfaces}.
Springer-Verlag, Berlin, 1988.

\bibitem{vwamelen:examples}
van Wamelen, P. B.,
``Examples of genus two cm curves defined over the rationals.''
\textit{Math. Comp.}, 68(225):307--320, 1999.

\bibitem{vwamelen:provingcm}
van Wamelen, P. B.,
``Proving that a genus 2 curve has complex multiplication.''
\textit{Math. Comp.}, 68(228):1663--1677, 1999.

\bibitem{vwamelen:analyticjacobians}
van Wamelen, P. B.,
``Computing with the analytic jacobian of a genus 2 curve.''
In \textit{Discovering mathematics with Magma}, pages 117--135.
Springer, 2006.

\bibitem{veech:ngon}
Veech, W. A.,
``Teichm\"uller curves in moduli space, Eisenstein series and an
application to triangular billiards.''
\textit{Invent. Math.}, 97(3):553--583, 1989.

\bibitem{veech:billiards}
Veech, W. A.,
``The billiard in a regular polygon.''
\textit{Geom. Funct. Anal.}, 2(3):341--379, 1992.

\bibitem{zorich:flatsurfaces}
Zorich, A.,
``Flat surfaces.''
In \textit{Frontiers in Number Theory, Physics and Geometry.
  Volume 1: On random matrices, zeta functions and dynamical systems},
pages 439--586. Springer-Verlag, Berlin, 2006.

\end{thebibliography}

\footnotesize{
\noindent\textsc{Department of Mathematics, Stony Brook University, Stony Brook, NY 11794, USA} \\
\textit{E-mail address}: \texttt{thenav@gmail.com} \\
\textsc{Department of Mathematics, Rice University, 6100 Main St., Houston, TX 77005, USA} \\
\textit{E-mail address}: \texttt{ronen@rice.edu}
}
\end{document}